\documentclass[twoside]{amsart}
\usepackage{amssymb}
\usepackage{amsmath}
\usepackage{latexsym}

\numberwithin{equation}{section}
\theoremstyle{definition} 
\newtheorem{definition}{Definition}

\newtheorem{remark}{Remark}

\theoremstyle{plain}
\newtheorem{theorem}{Theorem}
\newtheorem{proposition}{Proposition}

\DeclareMathOperator{\End}{\mathrm{End}}

\DeclareMathOperator{\Ad}{Ad}
\DeclareMathOperator{\ACYC}{o^+}

\DeclareMathOperator{\tr}{Tr}

\DeclareMathOperator{\RAMAN}{\Psi}
\DeclareMathOperator{\FQDILOG }{\phi}

\newcommand{\ALG}{{\mathcal{A}_{N}}}

\newcommand{\ASPR}[3]{\prod_{#2\le #1\le #3}}

\newcommand{\CLA}{c_b}
\newcommand{\cla}{c_b}
\newcommand{\COMPLEXS}{\mathbb C}

\newcommand{\DEHN}{\mathsf D}
\newcommand{\DESPR}[3]{\prod_{#3\ge #1\ge #2}}

\newcommand{\ELIM}{\eta}
\newcommand{\EULER}{M}

\newcommand{\FUNCTOR}{\mathsf F}

\newcommand{\GEN}{\mathsf r}
\newcommand{\GENI}{\mathsf g}

\newcommand{\HNSITES}{N}

\newcommand{\IHG}[1]{{\Psi_#1}}
\newcommand{\IMUN}{\mathsf i}
\newcommand{\imun}{\mathsf i}
\newcommand{\INCL}{\kappa}
\newcommand{\INTEGERS}{\mathbb Z}

\newcommand{\la}{b}
\newcommand{\LC}{\mathsf{U}_{lc}}
\newcommand{\LCT}{\tilde{\mathsf{U}}_{lc}}
\newcommand{\lmat}{\mathsf L}

\newcommand{\MCG}{\mathcal M_\Sigma}
\newcommand{\MOM}{\mathsf p}

\newcommand{\NSITES}{{2N}}
\newcommand{\NTR}{2M}

\newcommand{\OBALG}{\mathcal A_{N}}

\newcommand{\PERMUTE}{\mathsf P}
\newcommand{\PGROUP}{\mathbb S}
\newcommand{\POS}{\mathsf q}
\newcommand{\PTOLEMY}{\mathsf T}

\newcommand{\QDILOG}{\varphi_{b}}
\newcommand{\QDILOGI}{\bar\varphi_{b}}
\newcommand{\qmat}{\mathsf Q}
\newcommand{\QPAR}{q}
\newcommand{\QTIL}{\bar q}

\newcommand{\REALS}{\mathbb R}

\newcommand{\ROTATE}{\mathsf A}
\newcommand{\rsmall}{w}

\newcommand{\SDIT}{\Delta_\Sigma}

\newcommand{\sfa}{\mathsf a}

\newcommand{\sfu}{\mathsf u}

\newcommand{\SHFL}{\mathsf G}
\newcommand{\SIM}{\mathsf W}
\newcommand{\SIMF}{\mathsf W_{\mathrm{f}}}
\newcommand{\SIME}{\mathsf W_{\mathrm{o}}}

\newcommand{\SURFACE}{\Sigma}

\newcommand{\transfa}{\mathsf t}
\newcommand{\TRIANGLES}{T}

\newcommand{\vsp}{\mathcal H}

\newcommand{\vv}{\mathsf w}
\newcommand{\ff}{\mathsf r}
\newcommand{\zetai}{\zeta_{\mathrm{inv}}}
\newcommand{\zetao}{\zeta_{\mathrm{o}}}
\newcommand{\gpd}{\mathcal{G}_\Sigma}

\begin{document}

\title{Discrete Liouville equation and Teichm\"uller theory}

\author{R.M. Kashaev}\thanks{
Work partially supported by FNS Grant No.~200020-121675}
\address{Section de math\'ematiques, Universit\'e de Gen\`eve,\\
 2-4 rue du Li\`evre, Case postale 64,
 1211 Gen\`eve 4, Suisse
}
\email{Rinat.Kashaev@unige.ch}
\date{October 2008}
\begin{abstract}
The  relationship (both classical and quantum mechanical) between the discrete Liouville equation and Teichm\"uller theory is reviewed.
\end{abstract}
\maketitle
\tableofcontents
\section{Introduction}

The Liouville equation \cite{liouville} is a partial differential equation of the form
\begin{equation} \label{liouville}
\frac{\partial^2\phi}{\partial z\partial\bar z}=\frac12e^{\phi},
        \end{equation}
        which has a number of applications both in mathematics and mathematical physics. For example, it describes surfaces of constant negative curvature, thus playing
indispensable role in uniformization theory of Riemannian surfaces of negative Euler characteristic \cite{poinc}.
        Indeed, let $p\colon \mathbb{H}\to \Sigma$ be a universal covering map for a hyperbolic surface $\Sigma$, where $\mathbb{H}$ is the upper half plane with the standard Poincar\'e metric $ds^2$, and $\sigma\colon U\to \mathbb{H}$, $U\subset\Sigma$, a local section of $p$. Then, the pull-back metric $\sigma^*ds^2$ in conformal form $e^\phi \vert dz\vert^2$, $z$ being a local complex coordinate on $U$, gives a solution $\phi$ of the Liouville equation on $U$.

  In theoretical physics, the Liouville equation is often considered in analytically continued form with $z=x+t$ and $\bar z=x-t$ as real independent coordinates. In this case, it takes the form of a classical equation of motion for a relativistic $1+1$-dimensional field theoretical system:
       \begin{equation} \label{liouville1}
\frac{\partial^2\phi}{\partial t^2}-\frac{\partial^2\phi}{\partial x^2}=-2e^{\phi}.
        \end{equation}
The invariance with respect to (holomorphic) re-parameterizations, associated with the diffeomorphism group of the circle, make the Liouville equation relevant to two-dimensional gravity \cite{jackiw} and Conformal Field Theory (CFT) \cite{BPZ}. It is also a basic ingredient in the theory of noncritical strings \cite{polyakov}. These are the reasons for the recent interest in the Liouville equation, especially its quantum theory \cite{curtho,dhoker,gernev1,gernev2}, and its (quantum) integrability properties \cite{blz1,blz2,blz3}.

There are few, seemingly different aspects of the relationship between the Liouville equation and Teichm\"uller theory. One such aspect is purely classical, which comes through the above mentioned uniformization theory of surfaces of negative Euler characteristic. It exhibits even more profound features when one considers perturbative approach to quantum Liouville theory \cite{takhtajan,takhtajan1,takhtajan2}.

Another aspect is purely quantum, and it originates from a conjecture of H.~Verlinde \cite{verlinde}, which states that there is a mapping class group equivariant isomorphism between the space of quantum states of the quantum Teichm\"uller theory on a given surface and the space of conformal blocks of the quantum Liouville theory on the same surface, see \cite{teschner} for the up to date situation with the Verlinder conjecture.

One more aspect of the connection of the Liouville equation to Teichm\"uller theory has been considered recently in the works \cite{fkv,fk,fv} through the consideration of a specific \emph{discretized} version of the Liouville equation both on classical and quantum levels. The discrete Liouville equation has the form
\begin{equation}\label{liouville3}
\chi_{m,n-1}\chi_{m,n+1}=(1+\chi_{m-1,n})(1+\chi_{m+1,n}),
\end{equation}
where the discrete "space-time" is represented by integer lattice $\mathbb{Z}^2$ and the dynamical field $\chi_{m,n}$ is a  strictly positive real function on this lattice. To see in what sense this is a discretized version of the Liouville equation, let us take a small positive parameter $\epsilon$ as the lattice spacing of the discretized space-time, and consider the combination
\[
\phi_\epsilon(x,t)=-\log(\epsilon^2\chi_{m,n})
\]
in the limit, where $\epsilon\to 0$, $m,n\to\infty$ in such a way that the products $x=m\epsilon$, and $t=n\epsilon$ are kept fixed. If a solution $\chi_{m,n}$ of the discrete Liouville equation is such that such a limit exists, then the limiting value $\phi_0(x,t)$ solves the dynamical version (\ref{liouville1}) of the Liouville equation.

It is worth repeating here the remark of the paper \cite{fv} that the discrete Liouville equation is among the simplest examples of $Y$-systems \cite{zam}, though it is not a $Y$-system for which is true the Zamolodchikov's periodicity conjecture.

There are also other interesting connections of (quantum) discrete integrable systems with (hyperbolic) geometry, see for example, \cite{bob,bms2,bms1,bs}.

The purpose of this exposition  is, following the works \cite{fkv,fk,fv,k}, to review the relationship of the discrete Liouville equation and the Teichm\"uller theory both classically and quantum mechanically.

{\bf Acknoledgements.} It is a pleasure to thank L.D.~Faddeev and A.Yu.~Volkov  in collaboration with whom some of the results described in this survey were obtained. I would also like to thank 
 V.V.~Bazhanov and V.V.~Mangazeev for valuable comments on the initial version of the paper.

\section{Classical discrete Liouville equation}

\subsection{Discretization from the Liouville formula}

The interpretation of solutions of the Liouville equation in terms of pull-backs of the Poincar\'e metric leads to the Liouville formula for a general solution of the dynamical version (\ref{liouville1}) of the Liouville equation
 \begin{equation}\label{liouville2}
 e^{\phi(x,t)}=-4\frac{f'(u_-)g'(u_+)}{(f(u_-)-g(u_+))^2},\quad u_\pm=x\pm t
 \end{equation}
 where $f(x)$, $g(x)$ are two arbitrary smooth functions on the real line.  From the  structure of the Liouville formula it follows that  it makes sense on entire real plane provided the functions $f(x)$, $g(x)$ satisfy the conditions
 \begin{equation}\label{conditions}
 f'(x)g'(y)<0, \quad f(x)\ne g(y),\quad \forall (x,y)\in \mathbb{R}^2.
 \end{equation}
Despite the fact that in the dynamical version (\ref{liouville1}) of the Liouville equation the complex analytical aspects of the  uniformization of hyperbolic surfaces are somehow lost, there is still an action of the group $PSL(2,\mathbb{R})$ on functions $f(x)$ and $g(x)$ given by transformations
\[
f(x)\mapsto\frac{af(x)+b}{cf(x)+d},\quad g(x)\mapsto\frac{ag(x)+b}{cg(x)+d}
\]
which leave unchanged the Liouville solution (\ref{liouville2}). One can show that the set of $PSL(2,\mathbb{R})$-orbits of pairs $(f,g)$ is in bijection with the set of solutions of the Liouville equation. For example, in the physically interesting case of periodic solutions
\[
\phi(x+L,t)=\phi(x,t)
\]
which correspond to a space-time of the topological type of cylinder $S^1\times \mathbb{R}$, the functions $f(x)$ and $g(x)$ are quasi-periodic:
\begin{equation}\label{quasi}
f(x+L)=\frac{af(x)+b}{cf(x)+d},\quad g(x+L)=\frac{ag(x)+b}{cg(x)+d}
\end{equation}
where the $PSL(2,\mathbb{R})$-matrix
\[
T[\phi]=\begin{pmatrix}
a&b\\
c&d
\end{pmatrix}
\]
is called the \emph{monodromy} matrix associated with a periodic solution $\phi$.

The basic idea of A.Yu.~Volkov behind the  discrete Liouville equation is to write first a $PSL(2,\mathbb{R})$-invariant finite difference analogue  of the Liouville formula~(\ref{liouville2}) by  replacing its right hand side by a cross-ratio of four shifted quantities $f(x-t\pm\epsilon)$, $g(x+t\pm\epsilon)$. Namely, if we define
\begin{equation}\label{volkov}
h_\epsilon(x,t)=-\frac{(f(u_-+\epsilon)-g(u_++\epsilon))(f(u_--\epsilon)-g(u_+-\epsilon))}
{(f(u_-+\epsilon)-f(u_--\epsilon))(g(u_++\epsilon)-g(u_+-\epsilon))},\quad u_{\pm}=x\pm t,
\end{equation}
where we assume conditions~(\ref{conditions}), then it is easily seen that the limit
\[
\ell(x,t)=\lim_{\epsilon\to0}\epsilon^{2}h_\epsilon(x,t)
\]
exists and the formula
\[
e^{\phi(x,t)}=\frac1{\ell(x,t)}
\]
coincides with the Liouville formula~(\ref{liouville2}). On the other hand, one observes that four shifted cross-ratios  $h_\epsilon(x\pm\epsilon,t)$, $h_\epsilon(x,t\pm\epsilon)$ depend on six values $f(x-t+k\epsilon)$, $g(x+t+k\epsilon)$, with $k\in\{0,\pm1\}$. Taking into account the $PSL(2)$-invariance of the cross-ratios, this means that among  those four cross-ratios only three are independent, and one identifies the following relation
\[
h_\epsilon(x,t+\epsilon)h_\epsilon(x,t-\epsilon)=
(1+h_\epsilon(x+\epsilon,t))(1+h_\epsilon(x-\epsilon,t))
\]
which, after the substitutions $x=m\epsilon$ and $t=n\epsilon$ takes the form of the discrete Liouville equation~(\ref{liouville3}) for the function
\begin{equation}\label{volkov1}
\chi_{m,n}=h_\epsilon(m\epsilon,n\epsilon),\quad (m,n)\in\mathbb{Z}^2.
\end{equation}
Notice, that one and the same pair of functions $(f,g)$ is used for constructing solutions of both the Liouville equation and its discrete counterpart.

When the quasi-periodicity conditions~(\ref{quasi}) are satisfied, function~(\ref{volkov}) is periodic:
\[
h_\epsilon(x+L,t)=h_\epsilon(x,t),
\]
so that we have periodic solutions of the discrete Liouville equation as soon as the lattice spacing $\epsilon$ is chosen to be a rational multiple of the period $L$. Namely, for $\epsilon=LM/N$ with positive mutually prime integers $M$ and $N$, function~(\ref{volkov1}) satisfies the equation
\[
\chi_{m+N,n}=\chi_{m,n}.
\]
In particular, when $N=1$, $\chi_{m,n}$ is independent of the first argument, and the discrete Liouville equation becomes a one dimensional discrete equation of the form
\[
\zeta_{n-1}\zeta_{n+1}=(1+\zeta_n)^2.
\]
This latter equation has first appeared in this context in \cite{fv} where it has been interpreted as the evolution of the "zero-modes" of the continuous Liouville equation.

\subsection{Discrete Liouville equation and Teichm\"uller space}

We consider an annulus with $N$ marked points on each of its
boundary components ($2N$ points in total), labeled $A_1,\ldots, A_N$ for one boundary component, and $B_1,\ldots,B_N$, for another.
Additionally, choose an ideal triangulation shown in figure~\ref{fig:1}, where the variables $f_1,\ldots,f_{2N}$ not only serve to identify the interior edges, but  also denote the associated to the triangulation the  Fock coordinates in the Teichm\"uller space of the annulus.
\begin{figure}[htb]
\centering
\begin{picture}(200,60)(0,-10)
\put(0,0){\begin{picture}(120,40)
\multiput(0,0)(0,40){2}{\line(1,0){120}}
\multiput(0,0)(60,0){3}{\line(0,1){40}}
\multiput(0,40)(60,0){2}{\line(3,-2){60}}
\multiput(0,0)(60,0){3}{\circle*{3}}
\multiput(0,40)(60,0){3}{\circle*{3}}
\scriptsize
\put(1,15){$f_1$}
\put(25,15){$f_2$}
\put(61,15){$f_3$}
\put(85,15){$f_4$}
\put(160,15){$f_{2N}$}
\put(201,15){$f_{2N+1}$}

\end{picture}}

\put(140,0){\begin{picture}(60,40)
\multiput(0,0)(0,40){2}{\line(1,0){60}}
\multiput(0,0)(60,0){2}{\line(0,1){40}}
\multiput(0,40)(60,0){1}{\line(3,-2){60}}
\multiput(0,0)(60,0){2}{\circle*{3}}
\multiput(0,40)(60,0){2}{\circle*{3}}
\scriptsize
\end{picture}}

\multiput(125,0)(5,0){3}{\circle*{1}}
\multiput(125,40)(5,0){3}{\circle*{1}}
\scriptsize
\put(-5,-8){$A_1$}
\put(-5,42.5){$B_1$}
\put(55,-8){$A_2$}
\put(55,42.5){$B_2$}
\put(115,-8){$A_3$}
\put(115,42.5){$B_3$}
\put(135,-8){$A_{N}$}
\put(135,42.5){$B_{N}$}
\put(195,-8){$A_{N+1}$}
\put(195,42.5){$B_{N+1}$}
\end{picture}
\caption{An ideally triangulated annulus with $N$ marked
points on each boundary component.
The leftmost and the rightmost vertical
edges are identified with equalities $f_{2N+1}=f_1$, $A_{N+1}=A_1$, $B_{N+1}=B_1$.}\label{fig:1}
\end{figure}
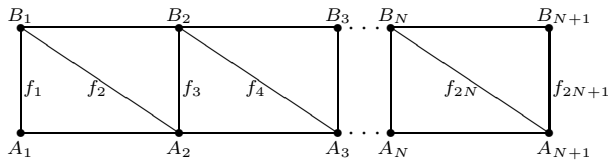
Notice that here we consider an unusual situation where all marked points are located on the boundary of the annulus, and for parametrization of the Teichm\"uller space one does not need to associate coordinates on the boundary ideal arcs. This can be seen by looking at a fundamental domain in the Poincar\'e upper plane which is an ideal polygon with $2N+2$ ideal vertices, see figure~\ref{fig:2} for the case $N=2$.

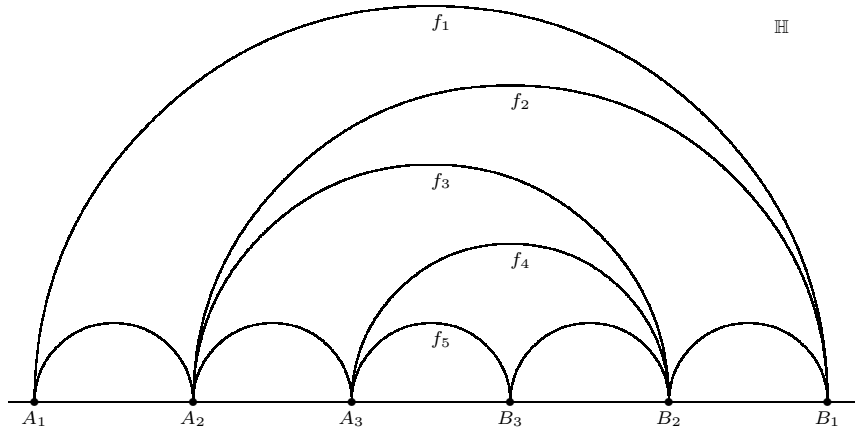
\begin{figure}[htb]
\centering
\begin{picture}(300,160)(0,-10)
\put(-10,0){\line(1,0){320}}
\qbezier(0,0)(0,62.132)(43.934,106.066)
\qbezier(43.934,106.066)(87.868,150)(150,150)
\qbezier(150,150)(212.132, 150)(256.066, 106.066)
\qbezier(256.066,106.066)(300, 62.132)(300,0)

\multiput(0,0)(60,0){5}{\qbezier(0., 0)( 0., 12.4264)( 8.7868, 21.2132)}
\multiput(0,0)(60,0){5}{\qbezier( 8.7868, 21.2132)( 17.5736, 30.)(30., 30.)}
\multiput(0,0)(60,0){5}{\qbezier( 30., 30.)(42.4264, 30.)( 51.2132, 21.2132)}
\multiput(0,0)(60,0){5}{\qbezier( 51.2132, 21.2132)(60., 12.4264)( 60., 0.)}

\qbezier(60., 0.)( 60., 49.7056)( 95.1472, 84.8528)
\qbezier( 95.1472, 84.8528)( 130.294, 120.)( 180., 120.)
\qbezier( 180., 120.)( 229.706, 120.)( 264.853, 84.8528)
\qbezier(264.853, 84.8528)( 300., 49.7056)( 300., 0)

\qbezier(60., 0.)( 60., 37.2792)( 86.3604, 63.6396)
\qbezier( 86.3604, 63.6396)( 112.721, 90.)( 150., 90.)
\qbezier(150., 90.)( 187.279, 90.)( 213.64, 63.6396)
\qbezier( 213.64, 63.6396)( 240., 37.2792)( 240., 0.)


\qbezier(120., 0.)( 120., 24.8528)(137.574, 42.4264)
\qbezier( 137.574, 42.4264)( 155.147, 60.)( 180., 60.)
\qbezier(180., 60.)( 204.853, 60.)( 222.426, 42.4264)
\qbezier( 222.426, 42.4264)( 240., 24.8528)( 240., 0.)
\scriptsize
\put(150,142){$f_1$}
\put(180,112){$f_2$}
\put(150,82){$f_3$}
\put(180,52){$f_4$}
\put(150,22){$f_5$}
\multiput(0,0)(60,0){6}{\circle*{3}}
\put(280,140){$\mathbb{H}$}
\put(-5,-8){$A_1$}
\put(55,-8){$A_2$}
\put(115,-8){$A_3$}
\put(175,-8){$B_3$}
\put(235,-8){$B_2$}
\put(295,-8){$B_1$}
\end{picture}
\caption{A fundamental domain in the Poincar\'e upper half plane for an ideally triangulated annulus with $N=2$ marked
points on each boundary component. Under the covering map $p\colon \mathbb{H}\to \mathbb{H}/\mathbb{Z}$ the images of the biggest circle with label $f_1$ and the middle smallest circle with label $f_5$ coincide so that $p(A_1)=p(A_3)$, $p(B_1)=p(B_3)$.}\label{fig:2}
\end{figure}
Clearly, the isometry class of such a polygon is determined by $2N-1$ real parameters corresponding to positions of $2N+2$ vertices modulo the 3-dimensional isometry group $PSL(2,\mathbb{R})$. This is less than the number of Fock coordinates $f_i$, $ 1\le i\le 2N$, but one more degree of freedom comes from the gluing condition: one chooses a $PSL(2,\mathbb{R})$-matrix restricted by the condition that it should map a given ordered pair of boundary points (the extremities $(A_1,B_1)$ of the half circle labeled $f_1$) to another ordered pair (the extremities $(A_{N+1},B_{N+1})$ of the half circle labeled $f_{2N+1}$), and it is known that there is a one parameter family of such matrices.

The mapping class group of our annulus is given by all homeomorphisms preserving the set of marked points, not necessarily point-wise. We are interested in the unique mapping class, denoted $D^{1/N}$, which fixes the set $\{A_1,\ldots,A_N\}$ point-wise and  cyclically permutes the
 set $\{B_1,\ldots,B_N\}$:
 \[
 B_1\mapsto B_2\mapsto B_3\mapsto\cdots\mapsto B_N\mapsto B_1.
 \]
 As the notation suggests, the $N$-th power of this class is nothing else than the only Dehn twist of the annulus which fixes the boundary point-wise.
 \begin{theorem}[\cite{fk}]
 The discrete dynamical system on the Teichm\"uller space of an annulus with $N$ marked points on each boundary component, corresponding to the mapping class $D^{1/N}$, is described by the discrete Liouville equation~(\ref{liouville3}) on the sublattice $m+n=1\pmod 2$ with the $2N$-periodic boundary condition
 \[
 \chi_{m+2N,n}=\chi_{m,n},
  \]
  the evolution step being identified with the translation along the "light-cone":
  \[
  \chi_{m,n}\mapsto\chi'_{m,n}=\chi_{m-1,n+1}.
  \]
 \end{theorem}
 \begin{proof}
 Recall that under a flip the Fock coordinates transform according to the formulas:
 \begin{equation}\label{flip}
 a'=a/(1+1/e),\  d'=d/(1+1/e),\  b'=b(1+e),\ c'=c(1+e),\ e'=1/e,
 \end{equation}
 where the variables are shown in Fig.~\ref{fig:3}, and all other variables staying unchanged.
 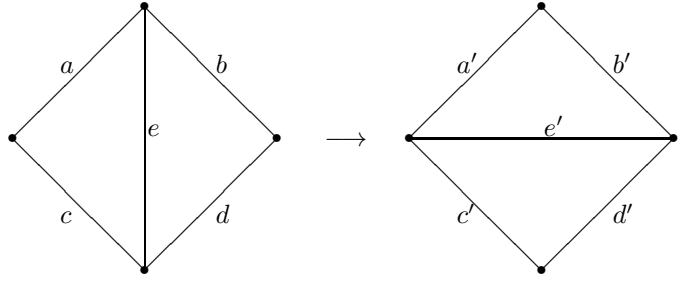
\begin{figure}[htb]
 \centering
 \begin{picture}(250,100)
 \multiput(0,0)(150,0){2}{\begin{picture}(100,100)
 \multiput(0,50)(50,-50){2}{\line(1,1){50}}
 \multiput(50,0)(50,50){2}{\line(-1,1){50}}
 \multiput(50,0)(50,50){2}{\circle*{3}}
 \multiput(0,50)(50,50){2}{\circle*{3}}\end{picture}}
 \put(50,0){\line(0,1){100}}
 \put(150,50){\line(1,0){100}}

 \put(118,47){$\longrightarrow$}

 \put(18,75){$a$}
 \put(77,75){$b$}
 \put(18,18){$c$}
 \put(77,18){$d$}
 \put(51,51){$e$}

 \put(168,75){$a'$}
 \put(227,75){$b'$}
 \put(168,18){$c'$}
 \put(227,18){$d'$}
 \put(201,51){$e'$}

 \end{picture}
 \caption{A flip transformation corresponding to equations~(\ref{flip}).}\label{fig:3}
 \end{figure}
We remark that this transformation law still applies even if some of the sides of the quadrilateral are a part of the boundary. The only modification is that there is no coordinate, associated to a boundary edge, and thus there is nothing to be transformed on this edge.

Now, from figure~\ref{fig:4} and the transformation law~(\ref{flip}) it follows that the mapping class $D^{1/N}$ acts in the Teichm\"uller space according to the following formulas
\begin{equation}\label{light-cone}
f_{2j}\mapsto f'_{2j}=1/f_{2j-1},\quad f_{2j+1}\mapsto f'_{2j+1}=f_{2j}(1+f_{2j-1})(1+f_{2j+1}).
\end{equation}
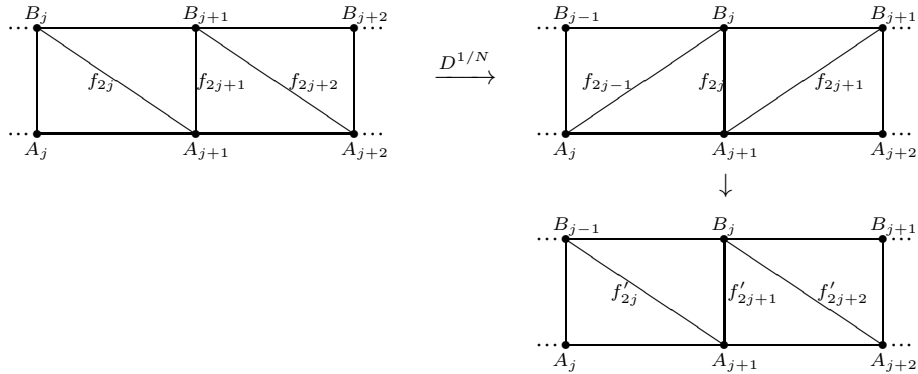
\begin{figure}[htb]
\centering
\begin{picture}(320,140)(0,-10)
\put(0,80){\begin{picture}(120,40)
\multiput(0,0)(0,40){2}{\line(1,0){120}}
\multiput(0,0)(60,0){3}{\line(0,1){40}}
\multiput(0,40)(60,0){2}{\line(3,-2){60}}
\multiput(0,0)(60,0){3}{\circle*{3}}
\multiput(0,40)(60,0){3}{\circle*{3}}
\scriptsize
\multiput(-4,0)(-3,0){3}{\circle*{1}}
\multiput(-4,40)(-3,0){3}{\circle*{1}}
\multiput(124,0)(3,0){3}{\circle*{1}}
\multiput(124,40)(3,0){3}{\circle*{1}}
\put(19,18){$f_{2j}$}
\put(60,18){$f_{2j+1}$}
\put(95,18){$f_{2j+2}$}
\put(-5,-8){$A_{j}$}
\put(-5,43.5){$B_{j}$}
\put(55,-8){$A_{j+1}$}
\put(55,43.5){$B_{j+1}$}
\put(115,-8){$A_{j+2}$}
\put(115,43.5){$B_{j+2}$}
\end{picture}}
\put(200,80){\begin{picture}(120,40)
\multiput(0,0)(0,40){2}{\line(1,0){120}}
\multiput(0,0)(60,0){3}{\line(0,1){40}}
\multiput(0,0)(60,0){2}{\line(3,2){60}}
\multiput(0,0)(60,0){3}{\circle*{3}}
\multiput(0,40)(60,0){3}{\circle*{3}}
\scriptsize
\multiput(-4,0)(-3,0){3}{\circle*{1}}
\multiput(-4,40)(-3,0){3}{\circle*{1}}
\multiput(124,0)(3,0){3}{\circle*{1}}
\multiput(124,40)(3,0){3}{\circle*{1}}
\put(6,18){$f_{2j-1}$}
\put(50,18){$f_{2j}$}
\put(94,18){$f_{2j+1}$}
\put(-5,-8){$A_{j}$}
\put(-5,43.5){$B_{j-1}$}
\put(55,-8){$A_{j+1}$}
\put(55,43.5){$B_{j}$}
\put(115,-8){$A_{j+2}$}
\put(115,43.5){$B_{j+1}$}
\end{picture}}
\put(200,0){\begin{picture}(120,40)
\multiput(0,0)(0,40){2}{\line(1,0){120}}
\multiput(0,0)(60,0){3}{\line(0,1){40}}
\multiput(0,40)(60,0){2}{\line(3,-2){60}}
\multiput(0,0)(60,0){3}{\circle*{3}}
\multiput(0,40)(60,0){3}{\circle*{3}}
\scriptsize
\multiput(-4,0)(-3,0){3}{\circle*{1}}
\multiput(-4,40)(-3,0){3}{\circle*{1}}
\multiput(124,0)(3,0){3}{\circle*{1}}
\multiput(124,40)(3,0){3}{\circle*{1}}
\put(17,18){$f'_{2j}$}
\put(61,18){$f'_{2j+1}$}
\put(95,18){$f'_{2j+2}$}
\put(-5,-8){$A_{j}$}
\put(-5,43.5){$B_{j-1}$}
\put(55,-8){$A_{j+1}$}
\put(55,43.5){$B_{j}$}
\put(115,-8){$A_{j+2}$}
\put(115,43.5){$B_{j+1}$}
\end{picture}}
\put(150,90){$\stackrel{D^{1/N}}{\overrightarrow{\phantom{D^{1/N}}}}$}

\put(258,58){$\downarrow$}
\end{picture}
\caption{The action of the mapping class $D^{1/N}$ on the triangulated
annulus: it is identical on the bottom boundary and a cyclic shift to the right by one spacing on the top boundary.}\label{fig:4}
\end{figure}
If we identify the variables $f_1,\ldots,f_{2N}$ with the initial data for the $2N$-periodic discrete Liouville equation~(\ref{liouville3}) on the sublattice $m+n=1\pmod 2$ along the zig-zag line $n\in\{-1,0\}$ according to formulas
\[
f_m=\left\{\begin{array}{cl}
\chi_{m,0}&\mathrm{if}\ m=1\pmod2;\\
1/\chi_{m,-1}&\mathrm{otherwise},
\end{array}
\right.
\]
then, the transformation formulas~(\ref{light-cone}) exactly correspond the the light-cone evolution:
\[
\chi_{m,n}\mapsto\chi'_{m,n}=\chi_{m-1,n+1}
\]
for the time instants $n\in\{-1,0\}$.
 \end{proof}

\section{Quantum theory}
In what follows we describe quantum discrete Liouville equation and its connection with quantum Teichm\"uller theory. Quantum theory of both the Teichm\"uller space and the discrete Liouville equation is build up on the base of a particular special function, called \emph{non-compact
quantum dilogarithm}. It is a building block of the operators realizing elements of the mapping class group in quantum Teichm\"uller theory, while in the case of quantum Liouville equation it is used for construction of the evolution operator and the Baxter's $Q$-operator (the operator reflecting the integrable structure of the discrete Liouville equation). This is why we start by describing some of the properties of this function.

\subsection{The non-compact quantum dilogarithm}

Let complex $\la$ have a nonzero real part $\Re \la\ne0$. Denote
\[
\CLA\equiv\IMUN(\la+\la^{-1})/2.
\]
The non-compact quantum dilogarithm, $\QDILOG(z)$, $z\in\COMPLEXS$, $|\Im z|<|\Im\CLA|$,
is defined by the formula
\begin{equation}\label{ncqdl}
\QDILOG(z)\equiv\exp\left(-\frac{1}{4}
\int_{-\infty}^{+\infty}
\frac{e^{-2\IMUN z x}\, dx}{\sinh(x\la)
\sinh(x/\la) x}\right),
\end{equation}
where the singularity at $x=0$ is put below the contour
of integration.
This definition implies that $\QDILOG(z)$ is unchanged under
substitutions $\la\to\la^{-1}$, $\la\to-\la$. Using this symmetry,
we choose $\la$ to lay in the first quadrant of the complex plane, namely
\[
\Re\la>0,\quad \Im\la\ge 0,
\]
which implies that $\Im\CLA>0$.
\begin{remark}
This function is closely related to double sine function of Barnes~\cite{barnes}. In the context of quantum integrable systems and quantum groups L.D.~Faddeev in \cite{faddeev95} pointed out its remarkable properties.
\end{remark}

\subsubsection{Functional relations}
In what follows, we shall use the following notation:
\begin{equation}\label{eq:zetas}
\zetai=e^{\IMUN\pi(1+2c_b^2)/6}=e^{\IMUN\pi c_b^2}\zetao^2,\quad
\zetao=e^{\IMUN\pi(1-4c_b^2)/12}.
\end{equation}
Function (\ref{ncqdl}) satisfies the `inversion' relation
\begin{equation}\label{inversion}
\QDILOG(z)\QDILOG(-z)=e^{-\IMUN\pi z^2}\zetai,
\end{equation}
and a pair of functional equations
\begin{equation}\label{shift}
\QDILOG(z+\IMUN\la^{\pm1}/2)=(1+e^{2\pi z\la^{\pm1} })
\QDILOG(z-\IMUN\la^{\pm1}/2).
\end{equation}
The latter equations enable us to extend definition of $\QDILOG(z)$ to the entire
complex plane.

When $\la$ is real or a pure phase,
function $\QDILOG(z)$ is {\em unitary} in the sense that
\begin{equation}\label{qdlunitarity}
(\QDILOG(z))^*=1/\QDILOG(z^*),\quad (1-|\la|)\Im\la=0.
\end{equation}
If self-adjoint operators $\MOM$ and $\POS$ in $L^2(\REALS)$ satisfy the
Heisenberg commutation relation
\begin{equation}\label{heisen}
\MOM\POS-\POS\MOM=(2\pi \IMUN)^{-1},
\end{equation}
then the following operator five term identity holds:
\begin{equation}\label{pent}
\QDILOG(\POS)\QDILOG(\MOM)=\QDILOG(\MOM)\QDILOG(\MOM+\POS)
\QDILOG(\POS).
\end{equation}
For real $\la$ this can be proved  in the $C^*$-algebraic framework \cite{wor}.
See \cite{fkv} for the proof in the case of complex $\la$ by the
use of the integral Ramanujan identity.

\subsubsection{Analytic properties}

Let $\Im\la^2>0$. We can perform the integration in (\ref{ncqdl})
by the residue method. The result can be written in the form
\begin{equation}\label{ratio}
\QDILOG(z)=(e^{2\pi (z-\CLA)\la^{-1}};\QTIL^2)_\infty/
(e^{2\pi (z+\CLA)\la};\QPAR^2)_\infty,
\end{equation}
where
\[
\QPAR=e^{\IMUN\pi\la^2},\quad \QTIL=e^{-\IMUN\pi\la^{-2}},
\]
and
\[
(x;y)_\infty\equiv\prod_{j=0}^\infty(1-xy^j),\quad x,y\in\COMPLEXS,\quad |y|<1.
\]
Formula~(\ref{ratio}) defines a meromorphic function on the entire complex
plane, satisfying functional equations (\ref{inversion}) and (\ref{shift}),
with essential singularity at infinity.
So, it is the desired extension of definition (\ref{ncqdl}).
It is easy to read off location of its poles and
zeroes:
\[
\mathrm{zeroes\ of\ }(\QDILOG(z))^{\pm1}=
\{\pm(\CLA+m\IMUN\la+n\IMUN\la^{-1}):\ m,n\in\INTEGERS_{\ge0}\}.
\]
The behavior at infinity depends on the direction along which the limit is
taken:
\begin{equation}\label{asymp}
\QDILOG(z)\bigg\vert_{|z|\to\infty} \approx\left\{
\begin{array}{ll}
1&|\arg(z)|>\frac{\pi}{2}+\arg(\la);\\
e^{-\IMUN\pi z^2}\zetai&|\arg(z)|<\frac{\pi}{2}-\arg(\la);\\
\Theta(\IMUN\la^{-1}z;-\la^{-2})/(\bar q^2;\bar q^2)_\infty&
|\arg z-\pi/2|<\arg\la;\\
(q^2; q^2)_\infty/\Theta(\IMUN\la z;\la^{2})&
|\arg z+\pi/2|<\arg\la,
\end{array}\right.
\end{equation}
where the standard notation for the $\Theta$-function is used:
\[
\Theta(z;\tau)\equiv\sum_{n\in\INTEGERS}e^{\IMUN\pi\tau n^2+2\pi\IMUN nz},
\quad\Im\tau>0.
\]
Thus, for complex $\la$, double quasi-periodic
$\theta$-functions, generators of the field of meromorphic functions on
complex tori, describe the asymptotic behavior of the non-compact quantum dilogarithm.

\subsubsection{Integral Ramanujan identity}

Consider the following Fourier integral:
\begin{equation}\label{ramanint}
\RAMAN(u,v,w)\equiv
\int_{\REALS}\frac{\QDILOG(x+u)}{\QDILOG(x+v)}e^{2\pi\IMUN wx}\, dx,
\end{equation}
where
\begin{equation}\label{restrictions1}
\Im(u+\CLA)>0,\quad\Im(-v+\CLA)>0, \quad \Im(u-v)<\Im w<0.
\end{equation}
Restrictions (\ref{restrictions1}) actually
can be considerably relaxed by deforming the
integration
path in the complex $x$ plane, keeping the asymptotic directions
of the two ends
within the sectors $\pm(|\arg x|-\pi/2)>\arg\la$. So, the enlarged in this way
 domain for the variables $u,v,w$ has the form:
\begin{equation}\label{restrictions2}
|\arg (\IMUN z)|<\pi-\arg\la,\quad z\in\{w,u-v-w,v-u-2\CLA\}.
\end{equation}
Integral (\ref{ramanint}) can be evaluated explicitly
by the residue method, the result being
\begin{gather}\label{ramanres1}
\RAMAN(u,v,w)=
\frac{\QDILOG(u-v+\CLA)\QDILOG(-w-\CLA)}{\QDILOG(u-v-w+\CLA)}
e^{-2\pi\IMUN w(v-\CLA)}\zetao^{-1}\\\label{ramanres2}
=\frac{\QDILOG(v+w-u-\CLA)}{\QDILOG(v-u-\CLA)\QDILOG(w+\CLA)}
e^{-2\pi\IMUN w(u+\CLA)}\zetao,
\end{gather}
where the two expressions in the right hand side are related to each other
through the inversion relation (\ref{inversion}). In \cite{fkv}
this identity has been demonstrated to be an integral counterpart of
the Ramanujan ${}_1\!\psi_1$ summation formula.

\subsubsection{Fourier transformations}

Particular values of $\RAMAN(u,v,w)$ lead to the following Fourier
transformation formulas:
\begin{multline}\label{fourier1}
\lefteqn{\FQDILOG_+(w)\equiv
\int_{\REALS}\QDILOG(x)e^{2\pi \IMUN wx}\,dx
=\RAMAN(0,v,w)\vert_{v\to-\infty}}\\
=e^{-2\pi\IMUN w\CLA}\zetao/
\QDILOG(w+\CLA)=
e^{\IMUN\pi w^2}\zetao^{-1}\QDILOG(-w-\CLA),\quad\quad
\end{multline}
and
\begin{multline}\label{fourier2}
\lefteqn{\FQDILOG_-(w)\equiv
\int_{\REALS}(\QDILOG(x))^{-1}e^{2\pi \IMUN wx}\,dx
=\RAMAN(u,0,w)\vert_{u\to-\infty}}\\
=e^{2\pi \IMUN w\CLA}\zetao^{-1}
\QDILOG(-w-\CLA)=e^{-\IMUN\pi w^2}\zetao/
\QDILOG(w+\CLA).\quad\quad
\end{multline}
The corresponding inverse transformations read:
\begin{equation}\label{finv}
(\QDILOG(x))^{\pm1}=\int_{\REALS}\FQDILOG_\pm(y)e^{-2\pi\IMUN xy}dy,
\end{equation}
where the pole at $y=0$ is surrounded from below.
\subsubsection{Other integral identities}

The non-compact quantum dilogarithm satisfies also integral analogs of other basic hypergeometric identities, see
for example  \cite{GR}.
For any $n\ge 1$ define
\begin{equation}\label{eq:ihg}
\IHG{n}(a_1,\ldots,a_n;b_1,\ldots,b_{n-1};w)
\equiv\int_\REALS dx\,
e^{\IMUN2\pi x(w-\cla)}\prod_{j=1}^n
\frac{\QDILOG(x+b_j-\cla)}{\QDILOG(x+a_j)},
\end{equation}
where $b_n=\IMUN0$,
\[
\Im(b_j)>0,\quad \Im(\cla-a_j)>0,\quad \sum_{j=1}^n\Im(b_j-a_j-\cla)<
\Im(w-\cla)<0.
\]
The integral analog of the  $_1\psi_1$-summation formula of Ramanujan in this notation takes the form
\begin{equation}\label{eq:raman}
\IHG{1}(a;w)=\zeta_o
\frac{\QDILOG(a+w-\cla)}{\QDILOG(a)
\QDILOG(w)}.
\end{equation}
Equivalently, we can rewrite it as follows
\begin{multline}\label{eq:ramanbar}
\bar\Psi_1(a;w)\equiv\int_\REALS\frac{\QDILOG(x+a)}{\QDILOG(x+\cla-\IMUN0)}
e^{-\IMUN2\pi x(w+\cla)}dx\\
=e^{\IMUN2\pi(a+\cla)(w+\cla)}\IHG{1}(-a;-w)
=\zeta_o^{-1}
\frac{\QDILOG(a)\QDILOG(w)}{\QDILOG(a+w+\cla)}.
\end{multline}
By using the integral Ramunujan formula, one can obtain an integral analog of the Heine transformation formula of the $_1\phi_2$ basic hypergeometric series:
\[
\IHG{2}(a,b;c;w)=\frac{\QDILOG(c-b)}{\QDILOG(a)}\,
\IHG{2}(c-b,w;a+w;b).
\]
By using the evident symmetry
\[
\IHG{2}(a,b;c;w)=\IHG{2}(b,a;c;w),
\]
we come to an integral analog of the Euler--Heine transformation formula
\begin{multline}\label{eq:eu-he}
\IHG{2}(a,b;c;w)\\
=\frac{\QDILOG(c-b)\QDILOG(c-a)\QDILOG(a+b+w-c)}{\QDILOG(a)\QDILOG(b)\QDILOG(w)}\,
\IHG{2}(c-a,c-b;c;a+b+w-c).
\end{multline}
Performing the Fourier transformation on the variable $w$ and using the equation~(\ref{eq:raman}), we obtain an integral analog of the summation formula of Saalsch\"utz:
\begin{multline}\label{eq:saal}
\IHG{3}(a,b,c;d,a+b+c-d-\cla;-\cla)
=\zeta_o^3
e^{\IMUN\pi d(2\cla-d)}\\
\times\frac{\QDILOG(a+b-d-\cla)\QDILOG(b+c-d-\cla)
\QDILOG(c+a-d-\cla)}{\QDILOG(a)\QDILOG(b)\QDILOG(c)
\QDILOG(a-d)\QDILOG(b-d)\QDILOG(c-d)}.
\end{multline}
One special case of this formula is obtained by taking the limit $c\to-\infty$:
\begin{equation}\label{eq:saal1}
\IHG{2}(a,b;d;-\cla)=\zeta_o^3
e^{\IMUN\pi d(2\cla-d)}\frac{\QDILOG(a+b-d-\cla)}{\QDILOG(a)\QDILOG(b)
\QDILOG(a-d)\QDILOG(b-d)}.
\end{equation}
From equation~(\ref{eq:saal}) one can derive an integral or non-compact analog of the Bailey lemma, which is equivalent to the following: if operators $\MOM$ and $\POS$
satisfy the Heisenberg commutation relation~(\ref{heisen}), then the operator valued function
\begin{multline}\label{eq:quv}
\qmat(u,v)=Q(u,v;\MOM,\POS)\\
\equiv e^{\IMUN\pi\POS^2}\QDILOG(u-\POS)\QDILOG(v-\POS)
\frac{\QDILOG(\MOM+u+v)}{\QDILOG(\MOM)}
\QDILOG(u+\POS)\QDILOG(v+\POS)e^{\IMUN\pi\POS^2}
\end{multline}
gives a commuting operator family in variables $u$ and $v$, and it acts diagonally on a one-parameter family
of vectors
\begin{equation}\label{eq:qeig}
\qmat(u,v)|\alpha_s\rangle=|\alpha_s\rangle \QDILOG(u+s)\QDILOG(v+s)
\QDILOG(u-s)\QDILOG(v-s)e^{\IMUN2\pi s^2},
\end{equation}
where the vectors $|\alpha_s\rangle$ are defined by their matrix elements
\begin{equation}
  \label{eq:al-s}
 \langle x|\alpha_s\rangle=\frac{\QDILOG(s-x-\cla+\IMUN 0)}{\QDILOG(s+x+\cla-\IMUN
  0)}e^{-\IMUN 2\pi(x+\cla)s},
\quad s\in\REALS_{\ge0},
\end{equation}
  with respect to the "position" basis
 $\langle x|$, $x\in\REALS$, where the operator $\POS$ is diagonal, and $\MOM$ acts as a differentiation:
\[
\langle x| \POS =x\langle x|,\quad\langle x| \MOM
=\frac1{2\pi\IMUN}\frac\partial{\partial x}\langle x|.
\]
Some of these and other interesting properties of the non-compact quantum dilogarithm as well as their interpretation in the context of integrable systems are also described in \cite{volkov1, bms1}.

\subsection{Quantum discrete Liouville equation}

The quantum version of the equation~(\ref{liouville3}) (with $\NSITES$-periodic boundary conditions) reads as
\begin{equation}\label{liouv}
\vv_{m,t+1}\vv_{m,t-1}=(1+q\vv_{m+1,t})(1+q^{1+2\delta_{\HNSITES,1}}\vv_{m-1,t}),
\end{equation}
where the field variables $\vv_{m,t}$
 are elements of the  \emph{observable algebra} (see below), satisfying the periodicity condition
 \[
 \vv_{m+\NSITES,t}=\vv_{m,t},
 \]
 and
$q=\exp(\IMUN\pi\la^2)$,
$\la$ being the {\em coupling} constant (or square root thereof). The latter
is expected to be related with the Virasoro central charge through the formula
\[
c_{\mathrm{Vir}}=1+6(\la+\la^{-1})^2.
\]
\begin{remark}
Notice that the case $\HNSITES=1$ is special, where the two terms in the right hand side of the equation~(\ref{liouv}) are given in terms of one and the same value of the field variable $\vv_{m+1,t}=\vv_{m-1,t}$. This is similar to the affine Cartan matrix of the type $A^{(1)}_N$ for $N=1$. A modification shows up also in the defining commutation relations of the observable algebra, see relations~(\ref{comrel}) below.
\end{remark}

 \subsection{Algebra of observables and the evolution operator}

The algebra of observables is generated by a finite set of self-adjoint
operators $\{\ff_1,\ldots, \ff_{\NSITES}\}$. We shall think of them as an operator family parameterized by integers $\{\ff_j\}_{j\in\mathbb{Z}}$  satisfying
the periodicity condition
\begin{equation}\label{percon}
\ff_{j+\NSITES}=\ff_j.
\end{equation}
The defining commutation relations are as follows
\begin{equation}\label{comrel}
[\ff_m,\ff_n]=\left\{\begin{array}{cl}
(-1)^m(1+\delta_{\HNSITES,1})(2\pi\IMUN)^{-1},&\ \mathrm{if}\ n=m\pm1\pmod {\NSITES};\\
0,&\ \mathrm{otherwise}.
\end{array}\right.
\end{equation}
Taking into account the interpretation in terms of the Teichm\"uller space of an annulus (with $\HNSITES$ marked points on each boundary component), the operators $\ff_j$ can be considered as quantized logarithmic Fock coordinates, the commutation relations~(\ref{comrel}) exactly corresponding to the Poisson structure of the Teichm\"uller space.

The initial data for the field variables in (\ref{liouv})
 are exponentials
of the generating elements:
\[
\vv_{2j+1,0}=e^{2\pi\la\ff_{2j+1}},\quad\vv_{2j,-1}=e^{2\pi\la\ff_{2j}}.
\]
\begin{proposition}
Let operator $\LC$ be defined by the formula
\begin{equation}
  \label{eq:evol-op}
 \LC=\SHFL\prod_{j=1}^N\QDILOGI(\ff_{2j}), \quad \QDILOGI(x)=1/\QDILOG(x),
\end{equation}
where operator $\SHFL$ is defined through the system of linear equations
\begin{equation}
  \label{eq:shiftFlip}
\SHFL\ff_j=(-1)^{j}\ff_{j-1}\SHFL,\quad j\in\mathbb{Z}.
\end{equation}
Then, the field variables defined by the formula
\begin{equation}
  \label{eq:field-var}
\vv_{j,t}\equiv \LC^te^{2\pi\la\ff_{j+t}}\LC^{-t},\quad j+t=1\pmod2,
\end{equation}
satisfy the $\NSITES$-periodic quantum discrete Liouville equation~(\ref{liouv}).
\end{proposition}
\begin{remark}
Note that the operator $\LC$ is identified with the "light-cone" evolution operator in the sense that
it realizes the appropriate translation in the "space-time" lattice through the formula
\[
\LC\chi_{j,t}\LC^{-1}=\chi_{j-1,t+1}.
\]
\end{remark}
\begin{remark}
Because of the duality symmetry
$\la\leftrightarrow\la^{-1}$ in the theory, there  actually
exist two types of exponential fields $\exp(2\pi\la^{\pm1} \ff_m)$
which satisfy two dual quantum discrete Liouville equations.
\end{remark}

\subsection{Integrable structure of the quantum discrete Liouville equation} The quantum discrete Liouville equation is integrable in the sense of the quantum inverse scattering method \cite{faddeev1}. That means that it admits a set of commuting operators with the evolution operator being one of them, and there is a system of linear difference equations, called Baxter equations, relating  these operators.

In what follows, the order in products with non-commuting entries will be indicated as
follows:
\[
\DESPR{i}{m}{n} a_i\equiv a_na_{n-1}\cdots a_{m+1}a_m,\quad
\ASPR{i}{m}{n} a_i\equiv a_ma_{m+1}\cdots a_{n-1}a_n.
\]
Consider algebra $\ALG$ of operators with a generalized linear basis of the form
\[
\ASPR{i}{1}{\NSITES}
e^{2\pi\IMUN\ff_i x_i},
\]
where self-adjoint operators $\ff_i$
satisfy  commutation relations (\ref{comrel}), and variables $x_i$
take real or complex values. The term "generalized" here means that a generic element of the algebra $\ALG$ is an integral of the form
\[
\int_{X^{\NSITES}} f(x_1,\ldots,x_{\NSITES})\left(\ASPR{i}{1}{\NSITES}
e^{2\pi\IMUN\ff_i x_i}\right)dx_1\cdots dx_{\NSITES},
\]
where $f(x_1,\ldots,x_{\NSITES})$ is a complex valued distribution (generalized function), and
$X^{\NSITES}\subset\mathbb{C}^{\NSITES}$ is a $\NSITES$-dimensional (over $\mathbb{R}$) sub-manifold.

The  ascending {\em cyclic product} is a set of linear
mappings,
\[
\ACYC_{j}\colon \ALG\rightarrow\ALG,\quad j\in\INTEGERS,\quad
\ACYC_{j}=\ACYC_{j+\NSITES},
\]
acting diagonally on the basis monomials:
\[
\ACYC_{1}(\ASPR{i}{1}{\NSITES}
e^{2\pi\IMUN\ff_i x_i})\equiv e^{2\pi\IMUN x_\NSITES x_1}
 \ASPR{i}{1}{\NSITES}
e^{2\pi\IMUN\ff_i x_i}\equiv\ACYC_{j}(\ASPR{i}{j}{ j+\NSITES-1}
e^{2\pi\IMUN\ff_i x_i}).
\]
We define the "transfer-matrices"
\begin{equation}\label{tmat}
\transfa^\pm(\mu)=\ACYC_1\tr\ASPR{j}{1}{\NSITES} \lmat^\pm_j,
\end{equation}
where
\begin{equation}\label{lop}
\lmat^\pm_j=
\left(\begin{array}{cc}
e^{(-1)^j\pi\la^{\pm1}(\mu-\ff_j)}&e^{-(-1)^j\pi\la^{\pm1}(\mu+\ff_j)}[j+1]_2\\
e^{-(-1)^j\pi\la^{\pm1}(\mu-\ff_j)}&e^{(-1)^j\pi\la^{\pm1}(\mu+\ff_j)}
\end{array}\right),
\end{equation}
\[
[j]_2=(1-(-1)^j)/2,
\]
and the trace is that of two-by-two matrices. We also define a "$Q$-operator"
\begin{equation}\label{qmat}
\qmat(\mu)
=\ACYC_1\left(\ASPR{j}{1}{\NSITES}
\rsmall_{[j]_2}(\mu,\ff_j)\right)\SHFL,\quad \mu\in\mathbb{C,}
\end{equation}
where
\begin{equation}\label{rmat}
e^{-2\pi\IMUN\mu x}\QDILOG(x-\mu)\rsmall_i(\mu,x)\equiv\left\{\begin{array}{cl}
\QDILOG(x+\mu)&,\ \mathrm{if}\ i=0;\\
1&,\ \mathrm{if}\ i=1,
\end{array}\right.
\end{equation}
and $\SHFL$ is  defined by equation~(\ref{eq:shiftFlip}).
\begin{proposition}
The transfer-matrices~(\ref{tmat}) and the $Q$-operator~(\ref{qmat}) commute among themselves
\begin{equation}\label{commut}
[\transfa^\epsilon(\mu),\transfa^\pm(\nu)]=[\qmat(\mu),\qmat(\nu)]=
[\transfa^\pm(\mu),\qmat(\nu)]=0, \quad \epsilon=\pm,
\end{equation}
solve the following Baxter equations:
\begin{equation}\label{baxeq}
\transfa^\pm(\mu)\qmat(\mu)=\qmat(\mu+\IMUN\la^{\pm1}/2)+
(1-e^{-4\pi\la^{\pm1}\mu})^\HNSITES\qmat(\mu-\IMUN\la^{\pm1}/2),
\end{equation}
and the evolution operator $\LC$ of the quantum discrete Liouville equation is given by the formula
\begin{equation}\label{eq:lcq}
\LC=\qmat(0).
\end{equation}
\end{proposition}
Formula~(\ref{eq:lcq}) is verified straightforwardly,
the commutativity part of the proposition is the standard argument by using the Yang--Baxter equations, while the proof of the Baxter equations given in \cite{k} uses a less standard argument.
\begin{remark}
The product of two neighboring $L$-operators
$\lmat^+_{2i}\lmat^+_{2i+1}$
is equivalent to the spectral parameter dependent
$L$-operator introduced in \cite{tirkFad}
for the description of the
(continuous) Liouville equation in the framework of the inverse scattering method.
\end{remark}

\subsection{The case $\HNSITES=1$}

When $\HNSITES=1$, the algebra $\ALG$ is generated by a single Heisenberg pair of position and momentum operators $\MOM$ and $\POS$:
\[
\GEN_1=-\MOM-\POS,\quad \GEN_2=-\MOM+\POS,\quad [\MOM,\POS]=(2\pi\IMUN)^{-1}.
\]
Calculation of the operators \eqref{tmat} and \eqref{qmat} at $\HNSITES=1$ gives the following result:
\[
\transfa(z)=L^+(\MOM,\POS)\equiv e^{2\pi\la\MOM}+2\cosh(2\pi\la\POS),
\]
\[
\zetao^3\qmat\left(\frac{z-\cla}2\right)=\int_{\REALS}
e^{\IMUN2\pi\cla(x+\cla)}
Q(-x-\cla+\IMUN0,-\infty;\MOM,\POS)e^{\IMUN2\pi xz}dx
\]
where $\zetao$ is defined in equation~(\ref{eq:zetas}) and
\[
Q(u,-\infty;\MOM,\POS)=e^{\IMUN\pi\POS^2}\QDILOG(u+\POS)
\QDILOGI(\MOM)\QDILOG(u-\POS)e^{\IMUN\pi\POS^2},
\]
see also equation~\eqref{eq:quv}. Calculation of the integral gives the formula
\[
\qmat(0)=\QDILOG(\MOM+\POS)
e^{\IMUN2\pi(\cla^2-\POS^2)}.
\]
By acting on the vectors $|\alpha_s\rangle$, and using equation~\eqref{eq:qeig}, we obtain
\begin{equation}\label{eq:tmoneq}
\transfa(z)|\alpha_s\rangle=L^+(\MOM,\POS)|\alpha_s\rangle=
|\alpha_s\rangle 2\cosh(2\pi\la s)
\end{equation}
and
\begin{equation}\label{eq:qmoneq}
\qmat\left(\frac{z-\cla}2\right)|\alpha_s\rangle=|\alpha_s\rangle
\zeta_o^{-1}e^{\IMUN\pi (s^2+2\cla^2)}[z|\alpha_s\rangle,
\end{equation}
where the vector
\[
[z|\equiv\int_{\REALS}dx
e^{\IMUN2\pi xz-\IMUN\pi x^2}\langle x|
\]
is such that
\begin{gather}\label{eq:inter}
[z|\MOM=p_z[z|,\quad [z|\POS=q_z[z|,\\
p_z\equiv\frac1{2\pi\IMUN}\frac{\partial}{\partial z}-z,\quad
q_z\equiv\frac1{2\pi\IMUN}\frac{\partial}{\partial z}.
\end{gather}
On the other hand, the Baxter equation~(\ref{baxeq}) at $\HNSITES=1$
can formally be written in the form
\begin{equation}\label{eq:nbe}
(L^+(p_z,q_z)-L^+(\MOM,\POS))\qmat\left(\frac{z-\cla}2\right)=0
\end{equation}
which, when applied to the vector $|\alpha_s\rangle$,
is reduced to an identity by using equations~(\ref{eq:tmoneq})~---
(\ref{eq:inter}).
\section{Relation to quantum Teichm\"uller theory}
\subsection{Highlights of quantum Teichm\"uller theory}
\subsubsection{Groupoid of decorated ideal triangulations}
Let $\Sigma=\Sigma_{g,s}$ be an oriented surface of genus $g$ with $s$ punctures. Denote $\EULER=2g-2+s$ and assume that $\EULER s>0$. Then surface $\Sigma$ admits ideal triangulations.
\begin{definition}
A \emph{decorated ideal triangulation} of $\Sigma$ is an ideal triangulation
$\tau$, where all triangles are provided with a marked corner, and a bijective ordering map
    \[
\bar\tau\colon
\{1,\ldots,2\EULER\}\ni j\mapsto\bar\tau_j\in\TRIANGLES(\tau)
    \]
is fixed. Here $\TRIANGLES(\tau)$ is the set of all triangles of $\tau$.
\end{definition}
Graphically,  the marked corner of a triangle is indicated by an asterisk and the corresponding number is put inside the triangle. The set of all decorated ideal triangulations of  $\SURFACE$ is denoted
$\SDIT$.

Recall that if a group $G$ freely acts in a set $X$ then there is an associated groupoid defined as follows. The objects are the $G$-orbits in $X$, while morphisms are $G$-orbits in $X\times X$ with respect to the diagonal action. Denote by $[x]$ the object represented by element $x\in X$ and $[x,y]$ the morphism represented by pair of elements $(x,y)\in X\times X$. Two morphisms $[x,y]$ and $[u,v]$, are composable if and only if $[y]=[u]$ and their composition is  $[x,y][u,v]=[x,gv]$, where $g\in G$ is the unique element sending $u$ to $y$. The inverse and the identity morphisms are given respectively by $[x,y]^{-1}=[y,x]$ and $\mathrm{id}_{[x]}=[x,x]$.

Remarking that the mapping class group $\MCG$ of $\Sigma$ freely acts in $\SDIT$, denote by $\gpd$ the corresponding groupoid, called the \emph{groupoid of decorated ideal triangulations}. It admits a presentation with three types of generators and four types of relations.

The generators are of the form $[\tau,\tau^\sigma]$, $[\tau,\rho_i\tau]$, and $[\tau,\omega_{i,j}\tau]$, where $\tau^\sigma$ is obtained from $\tau$ by replacing the ordering map $\bar\tau$ by the map $\bar{\tau}\circ\sigma$, where $\sigma\in\mathbb{S}_{2\EULER}$ is a permutation of the set $\{1,\ldots,2\EULER\}$, $\rho_i\tau$ is obtained from $\tau$ by changing the marked corner of triangle $\bar\tau_i$ as in figure~\ref{ft}, and $\omega_{i,j}\tau$ is obtained from $\tau$ by applying the flip transformation in the quadrilateral composed of triangles $\bar\tau_i$ and $\bar\tau_j$ as in figure~\ref{fe}.
\begin{figure}[htb]
  \centering
\begin{picture}(200,20)
\put(0,0){\begin{picture}(40,20)
\put(0,0){\line(1,0){40}}
\put(0,0){\line(1,1){20}}
\put(20,20){\line(1,-1){20}}
\put(0,0){\circle*{3}}
\put(20,20){\circle*{3}}
\put(40,0){\circle*{3}}
\footnotesize
\put(33,0){$*$}
\put(18,5){$i$}
\end{picture}}
\put(160,0){\begin{picture}(40,20)
\put(0,0){\line(1,0){40}}
\put(0,0){\line(1,1){20}}
\put(20,20){\line(1,-1){20}}
\put(0,0){\circle*{3}}
\put(20,20){\circle*{3}}
\put(40,0){\circle*{3}}
\footnotesize
\put(17.5,14){$*$}
\put(18,5){$i$}
\end{picture}}
\put(95,8){$\stackrel{\rho_i}{\longrightarrow}$}
\end{picture}
\caption{Transformation $\rho_i$.}\label{ft}
\end{figure}
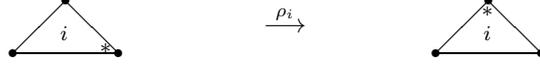
\begin{figure}[htb]
  \centering
\begin{picture}(200,40)
\put(0,0){
\begin{picture}(40,40)
\put(20,0){\line(-1,1){20}}
\put(40,20){\line(-1,-1){20}}
\put(0,20){\line(1,1){20}}
\put(40,20){\line(-1,1){20}}
\put(20,0){\line(0,1){40}}
\put(20,0){\circle*{3}}
\put(0,20){\circle*{3}}
\put(20,40){\circle*{3}}
\put(40,20){\circle*{3}}
\footnotesize
\put(10,18){$i$}\put(26,18){$j$}
\put(1,18){$*$}
\put(19.5,2){$*$}
\end{picture}}
\put(160,0){\begin{picture}(40,40)
\put(20,0){\line(-1,1){20}}
\put(40,20){\line(-1,-1){20}}
\put(0,20){\line(1,1){20}}
\put(40,20){\line(-1,1){20}}
\put(0,20){\line(1,0){40}}
\put(20,0){\circle*{3}}
\put(0,20){\circle*{3}}
\put(20,40){\circle*{3}}
\put(40,20){\circle*{3}}
\footnotesize
\put(18,26){$i$}\put(18,10){$j$}
\put(3,20){$*$}
\put(17.5,1){$*$}
\end{picture}}
\put(95,17){$\stackrel{\omega_{i,j}}{\longrightarrow}$}
\end{picture}
\caption{Transformation $\omega_{i,j}$.}\label{fe}
\end{figure}
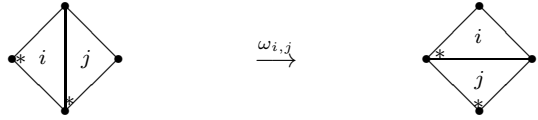

These generators satisfy the following relations:
\begin{gather}
\label{eq:23}
[\tau,(\tau^\alpha)^\beta]=[\tau,\tau^{\alpha\beta}],\quad \alpha,\beta\in\mathbb{S}_{2\EULER},\\
  \label{eq:19}
  [\tau,\rho_i\rho_i\rho_i\tau]=\mathrm{id}_{[\tau]},\\\label{eq:22}
  [\tau,\omega_{j,k}\omega_{i,k}\omega_{i,j}\tau]
  =[\tau,\omega_{i,j}\omega_{j,k}\tau],
\\
\label{eq:20}
[\tau,\omega_{j,i}\rho_i\omega_{i,j}\tau]=[\tau,\rho_i\rho_j\tau^{(ij)}].
\end{gather}
The first two relations are evident, while the other two are shown graphically in figures~\ref{fig:pen-om},~\ref{fig:inv-rel}.
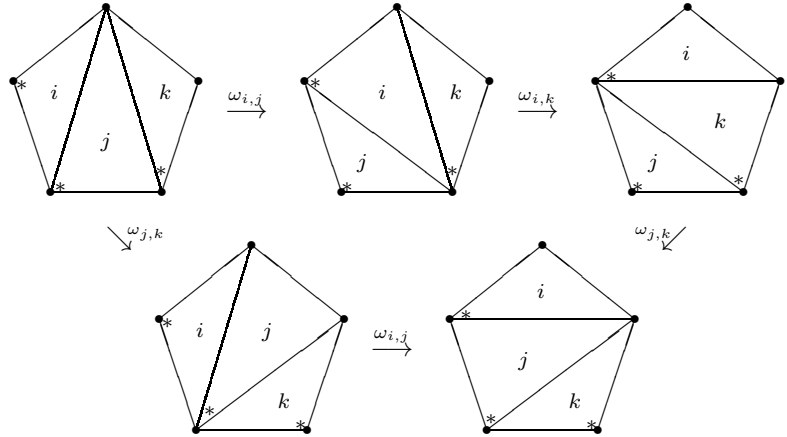
\begin{figure}[htb]
  \centering
\begin{picture}(290,160)
\put(0,90){\begin{picture}(70,70)
\put(14,0){\line(-1,3){14}}
\put(56,0){\line(1,3){14}}
\put(0,42){\line(5,4){35}}
\put(70,42){\line(-5,4){35}}
\put(14,0){\line(1,0){42}}
\qbezier(14,0)(20,20)(35,70)
\qbezier(56,0)(50,20)(35,70)
\put(14,0){\circle*{3}}
\put(56,0){\circle*{3}}
\put(0,42){\circle*{3}}
\put(70,42){\circle*{3}}
\put(35,70){\circle*{3}}
\footnotesize
\put(1,38.5){$*$}
\put(15.5,0){$*$}
\put(53.75,6){$*$}
\put(14,35){$i$}
\put(33,17){$j$}
\put(55,35){$k$}
\end{picture}}
\put(110,90){\begin{picture}(70,70)
\put(14,0){\line(-1,3){14}}
\put(56,0){\line(1,3){14}}
\put(0,42){\line(5,4){35}}
\put(70,42){\line(-5,4){35}}
\put(14,0){\line(1,0){42}}
\put(56,0){\line(-4,3){56}}
\qbezier(56,0)(50,20)(35,70)
\put(14,0){\circle*{3}}
\put(56,0){\circle*{3}}
\put(0,42){\circle*{3}}
\put(70,42){\circle*{3}}
\put(35,70){\circle*{3}}
\footnotesize
\put(2,39.5){$*$}
\put(14,0){$*$}
\put(53.75,6){$*$}
\put(28,35){$i$}
\put(20,9){$j$}
\put(55,35){$k$}
\end{picture}}
\put(220,90){\begin{picture}(70,70)
\put(14,0){\line(-1,3){14}}
\put(56,0){\line(1,3){14}}
\put(0,42){\line(5,4){35}}
\put(70,42){\line(-5,4){35}}
\put(14,0){\line(1,0){42}}
\put(56,0){\line(-4,3){56}}
\put(0,42){\line(1,0){70}}
\put(14,0){\circle*{3}}
\put(56,0){\circle*{3}}
\put(0,42){\circle*{3}}
\put(70,42){\circle*{3}}
\put(35,70){\circle*{3}}
\footnotesize
\put(4,41.5){$*$}
\put(14,0){$*$}
\put(52.25,2.5){$*$}
\put(33,50){$i$}
\put(20,9){$j$}
\put(45,24 ){$k$}
\end{picture}}
\put(55,0){\begin{picture}(70,70)
\put(14,0){\line(-1,3){14}}
\put(56,0){\line(1,3){14}}
\put(0,42){\line(5,4){35}}
\put(70,42){\line(-5,4){35}}
\put(14,0){\line(1,0){42}}
\qbezier(14,0)(20,20)(35,70)
\put(14,0){\line(4,3){56}}
\put(14,0){\circle*{3}}
\put(56,0){\circle*{3}}
\put(0,42){\circle*{3}}
\put(70,42){\circle*{3}}
\put(35,70){\circle*{3}}
\footnotesize
\put(1,38.5){$*$}
\put(17,5){$*$}
\put(51.5,0){$*$}
\put(14,35){$i$}
\put(39,35){$j$}
\put(45,8){$k$}
\end{picture}}
\put(165,0){\begin{picture}(70,70)
\put(14,0){\line(-1,3){14}}
\put(56,0){\line(1,3){14}}
\put(0,42){\line(5,4){35}}
\put(70,42){\line(-5,4){35}}
\put(14,0){\line(1,0){42}}
\put(0,42){\line(1,0){70}}
\put(14,0){\line(4,3){56}}
\put(14,0){\circle*{3}}
\put(56,0){\circle*{3}}
\put(0,42){\circle*{3}}
\put(70,42){\circle*{3}}
\put(35,70){\circle*{3}}
\footnotesize
\put(4,41.5){$*$}
\put(13.5,2){$*$}
\put(51.5,0){$*$}
\put(33,50){$i$}
\put(26,24){$j$}
\put(45,8){$k$}
\end{picture}}
\put(35,70){$\searrow$}\put(43,74){\tiny$\omega_{j,k}$}
\put(245,70){$\swarrow$}\put(235,74){\tiny$\omega_{j,k}$}
\put(135,28){$\stackrel{\omega_{i,j}}{\longrightarrow}$}
\put(80,118){$\stackrel{\omega_{i,j}}{\longrightarrow}$}
\put(190,118){$\stackrel{\omega_{i,k}}{\longrightarrow}$}
\end{picture}
  \caption{Pentagon relation~\eqref{eq:22}.}
  \label{fig:pen-om}
\end{figure}

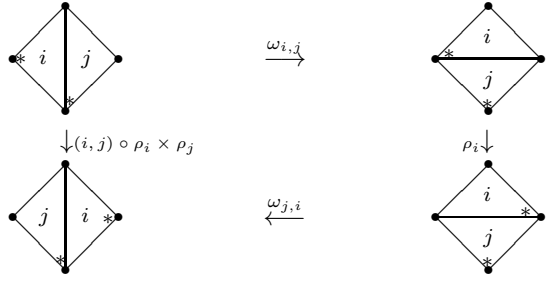
\begin{figure}[htb]
  \centering
  \begin{picture}(200,100)
\put(0,60){\begin{picture}(200,40)
\put(0,0){\begin{picture}(40,40)
\put(20,0){\line(-1,1){20}}
\put(40,20){\line(-1,-1){20}}
\put(0,20){\line(1,1){20}}
\put(40,20){\line(-1,1){20}}
\put(20,0){\line(0,1){40}}
\put(20,0){\circle*{3}}
\put(0,20){\circle*{3}}
\put(20,40){\circle*{3}}
\put(40,20){\circle*{3}}
\footnotesize
\put(10,18){$i$}\put(26,18){$j$}
\put(1,18){$*$}
\put(19.5,2){$*$}
\end{picture}}
\put(160,0){\begin{picture}(40,40)
\put(20,0){\line(-1,1){20}}
\put(40,20){\line(-1,-1){20}}
\put(0,20){\line(1,1){20}}
\put(40,20){\line(-1,1){20}}
\put(0,20){\line(1,0){40}}
\put(20,0){\circle*{3}}
\put(0,20){\circle*{3}}
\put(20,40){\circle*{3}}
\put(40,20){\circle*{3}}
\footnotesize
\put(18,26){$i$}\put(18,10){$j$}
\put(3,20){$*$}
\put(17.5,1){$*$}
\end{picture}}
\put(95,17){$\stackrel{\omega_{i,j}}{\longrightarrow}$}
\end{picture}}
\put(18,46){$\downarrow$\tiny$(i,j)\circ\rho_i\times\rho_j$}
\put(170,46){{\tiny$\rho_i$}$\downarrow$}
\put(0,0){\begin{picture}(200,40)
\put(0,0){\begin{picture}(40,40)
\put(20,0){\line(-1,1){20}}
\put(40,20){\line(-1,-1){20}}
\put(0,20){\line(1,1){20}}
\put(40,20){\line(-1,1){20}}
\put(20,0){\line(0,1){40}}
\put(20,0){\circle*{3}}
\put(0,20){\circle*{3}}
\put(20,40){\circle*{3}}
\put(40,20){\circle*{3}}
\footnotesize
\put(10,18){$j$}\put(26,18){$i$}
\put(16,2){$*$}
\put(34,17.5){$*$}
\end{picture}}
\put(160,0){\begin{picture}(40,40)
\put(20,0){\line(-1,1){20}}
\put(40,20){\line(-1,-1){20}}
\put(0,20){\line(1,1){20}}
\put(40,20){\line(-1,1){20}}
\put(0,20){\line(1,0){40}}
\put(20,0){\circle*{3}}
\put(0,20){\circle*{3}}
\put(20,40){\circle*{3}}
\put(40,20){\circle*{3}}
\footnotesize
\put(18,26){$i$}\put(18,10){$j$}
\put(32,20){$*$}
\put(17.5,1){$*$}
\end{picture}}
\put(95,17){$\stackrel{\omega_{j,i}}{\longleftarrow}$}
\end{picture}}
\end{picture}
  \caption{Inversion relation~\eqref{eq:20}.}
  \label{fig:inv-rel}
\end{figure}

\subsubsection{Hilbert spaces of square integrable functions}
\label{sec:q-functor}

In what follows, we work with Hilbert spaces
\[
\vsp\equiv L^2(\REALS),\quad \vsp^{\otimes n}\equiv L^2(\REALS^n).
\]
Any two self-adjoint operators
$\MOM$ and $\POS$, acting in $\vsp$ and satisfying the Heisenberg commutation relation~(\ref{heisen}),
can be realized as differentiation and multiplication operators. Such
"coordinate" realization in Dirac's bra-ket notation has the form
\begin{equation}\label{eq:mdo}
\langle x|\MOM =\frac1{2\pi\IMUN}\frac{\partial}{\partial x}\langle x|,
\quad \langle x|\POS =x\langle x|.
\end{equation}
Formally, the set of "vectors" $\{|x\rangle\}_{x\in\mathbb{R}}$ forms a generalized basis of $\vsp$ with the following orthogonality and completeness properties:
\[
\langle x|y\rangle=\delta(x-y),\quad \int_{\mathbb{R}}|x\rangle dx\langle x|=1.
\]
For any $1\le i\le m$ we shall use the following notation
        \[
        \iota_i\colon \End\vsp\ni \sfa\mapsto
        \sfa_i=\underbrace{1\otimes\cdots\otimes1}_{i-1\ \mathrm{times}}\otimes
        \sfa\otimes1\otimes\cdots\otimes1
        \in\End\vsp^{\otimes m}.
         \]
Besides that, if $\sfu\in \End\vsp^{\otimes k}$ for some $1\le
         k\le m$ and  $\{i_1,i_2,\ldots,i_k\}\subset\{1,2,\ldots,m\}$,
         then we shall write
    \[
\sfu_{i_1i_2\ldots i_2}\equiv\iota_{i_1}\otimes\iota_{i_2}\otimes\cdots
\otimes\iota_{i_k}(\sfu).
    \]
The permutation group $\PGROUP_m$ naturally acts in $\vsp^{\otimes m}$:
    \begin{equation}\label{eq:perm}
\PERMUTE_\sigma (x_1\otimes\cdots\otimes
x_i\otimes\cdots) =x_{\sigma^{-1}(1)}\otimes\cdots\otimes
x_{\sigma^{-1}(i)}\otimes\ldots,\quad\sigma\in \PGROUP_m.
    \end{equation}

\subsubsection{Semi-symmetric $T$-matrix}
\label{sec:inii-aeaa-neno}

Fix self-conjugate operators $\MOM,\POS$ satisfying the Heisenberg commutation relation~(\ref{heisen}). Choose a parameter $\la$ satisfying the condition
\[
(1-|\la|)\Im\la=0,
\]
 and define two unitary operators
\begin{gather}
  \label{eq:5}
  \ROTATE\equiv e^{-\IMUN\pi/3}e^{\IMUN
    3\pi\POS^2}e^{\IMUN\pi(\MOM+\POS)^2}\in \End\vsp,\\
\label{eq:t-in-t-of-psi}
\PTOLEMY\equiv e^{\IMUN 2\pi\MOM_1\POS_2}
\QDILOG(\POS_1+\MOM_2-\POS_2)\in \End\vsp^{\otimes2}.
    \end{gather}
They satisfy the following relations characterizing a \emph{semi-symmetric $T$-matrix}:
\begin{gather}
  \label{eq:1}
\ROTATE^3=1, \\
 \label{eq:2}
\PTOLEMY_{12}\PTOLEMY_{13}\PTOLEMY_{23}=\PTOLEMY_{23}\PTOLEMY_{12}, \\
  \label{eq:4}
\PTOLEMY_{12} \ROTATE_1\PTOLEMY_{21}=\zeta\ROTATE_1\ROTATE_2\PERMUTE_{(12)},
\end{gather}
where
\begin{equation}
  \label{eq:prfac}
  \zeta\equiv e^{\IMUN\pi\cla^2/3},\quad \cla=\frac\imun2(\la+\la^{-1}),
\end{equation}
 and operator $\PERMUTE_{(12)}$ is defined by equation~\eqref{eq:perm} in the case when
\[
\PGROUP_2\ni\sigma=(12)\colon 1\mapsto2\mapsto1.
\]
 Operator $\ROTATE$ is characterized (up to a normalization factor) by the equations
\[
\ROTATE\POS\ROTATE^{-1}=\MOM-\POS,\quad\ROTATE\MOM\ROTATE^{-1}=-\POS.
\]
Note that equations~(\ref{eq:1})---(\ref{eq:4}) correspond to relations~
\eqref{eq:19}---\eqref{eq:20}. This fact is the base of using the former to realize the latter.

\subsubsection{Useful notation}
\label{sec:useful-notation}

For any operator $\sfa\in\End \vsp$ we shall denote
\begin{equation}
  \label{eq:14}
  \sfa_{\hat k}\equiv \ROTATE_k\sfa_k\ROTATE_k^{-1},\quad
 \sfa_{\check k}\equiv \ROTATE_k^{-1}\sfa_k\ROTATE_k.
\end{equation}
It is evident that
\[
\sfa_{\check{\hat k}}=\sfa_{\hat{\check k}}=\sfa_k,\quad \sfa_{\hat{\hat
    k}}=\sfa_{\check k},\quad \sfa_{\check{\check k}}=\sfa_{\hat k},
\]
where the last two equations follow from equation~\eqref{eq:1}. In particular,
we have
\begin{gather}
\MOM_{\hat k}=-\POS_k,\quad\POS_{\hat k}=\MOM_k-\POS_k,\\
\MOM_{\check k}=\POS_k-\MOM_k,\quad\POS_{\check k}=-\MOM_k.
\end{gather}
Besides that, it will be also useful to use the notation
\begin{equation}
  \label{eq:16}
 \PERMUTE_{(kl\ldots m\hat k)}\equiv\ROTATE_k\PERMUTE_{(kl\ldots
  m)},\quad
\PERMUTE_{(kl\ldots m\check k)}\equiv\ROTATE_k^{-1}\PERMUTE_{(kl\ldots
  m)},
\end{equation}
where $(kl\ldots m)$ is the cyclic permutation
\[
(kl\ldots m)\colon k\mapsto l\mapsto\ldots\mapsto m\mapsto k.
\]
Equation~\eqref{eq:4} in this notation takes a rather compact form
\begin{gather}
\label{eq:16a}
  \PTOLEMY_{12}\PTOLEMY_{2\hat1}= \zeta\PERMUTE_{(12\hat1)}.
\end{gather}
\begin{remark}
One can derive the following symmetry property of the $T$-matrix:
\begin{multline*}
\PTOLEMY_{12}=\PTOLEMY_{12}\PTOLEMY_{2\hat1}\PTOLEMY_{2\hat1}^{-1}=
\zeta\PERMUTE_{(12\hat1)}\PTOLEMY_{2\hat1}^{-1}=
\PTOLEMY_{\hat1\hat2}^{-1}\zeta\PERMUTE_{(12\hat1)}\\=
\PTOLEMY_{\hat1\hat2}^{-1}\zeta\PERMUTE_{(\hat1\hat2\check1)}=
\PTOLEMY_{\hat1\hat2}^{-1}\PTOLEMY_{\hat1\hat2}\PTOLEMY_{\hat2\check1}=
\PTOLEMY_{\hat2\check1}.
\end{multline*}
\end{remark}

\subsubsection{Quantum functor}
\label{sec:eaaioiaue-ooieoid}

Quantum Teichm\"uller theory, being a three-dimensional TQFT, is defined by a \emph{quantum functor},
\[
\FUNCTOR\colon\gpd\to\End\vsp^{\otimes\NTR},
\]
which means that we have a operator valued function
\[
\FUNCTOR\colon\SDIT\times\SDIT\to \End\vsp^{\otimes\NTR},
\]
satisfying the equations
 \begin{equation}
  \label{eq:10}
 \FUNCTOR(\tau,\tau)=1,\quad \FUNCTOR(\tau,\tau')\FUNCTOR(\tau',\tau'')
\FUNCTOR(\tau'',\tau)\in\COMPLEXS\setminus\{0\},\quad \forall\tau,\tau',\tau''\in\SDIT,
\end{equation}
 \begin{equation}
\FUNCTOR(f(\tau),f(\tau'))= \FUNCTOR(\tau,\tau'),\qquad \forall
f\in\MCG,
\end{equation}
\begin{equation}
  \label{eq:17}
     \FUNCTOR(\tau,\rho_i\tau)\equiv\ROTATE_i,
\end{equation}
    \begin{equation}\label{tij}
\FUNCTOR(\tau,\omega_{i,j}\tau)\equiv\PTOLEMY_{ij},
    \end{equation}
\begin{equation}
\FUNCTOR(\tau,\tau^\sigma)\equiv\PERMUTE_\sigma,
\quad\forall\sigma\in\PGROUP_{\NTR},
    \end{equation}
where operator $\PERMUTE_\sigma$ is defined by equation~(\ref{eq:perm}).
Consistency of these equations is ensured by the consistency of equations~(\ref{eq:1})---(\ref{eq:4}) with relations~\eqref{eq:19}---\eqref{eq:20}.

A particular case of equation~(\ref{eq:10}) corresponds to
$\tau''=\tau$:
\begin{equation}
  \label{eq:18}
  \FUNCTOR(\tau,\tau')\FUNCTOR(\tau',\tau)\in\COMPLEXS\setminus\{0\}.
\end{equation}
As an example, we can calculate the operator
$\FUNCTOR(\tau,\omega_{i,j}^{-1}(\tau))$. Denoting $\tau'\equiv
\omega_{i,j}^{-1}(\tau)$ and using equation~(\ref{eq:18}), as well as definition~\eqref{tij}, we obtain
\begin{equation}
  \label{eq:11}
 \FUNCTOR(\tau,\omega_{i,j}^{-1}(\tau))=
\FUNCTOR(\omega_{i,j}(\tau'),\tau')\simeq
(\FUNCTOR(\tau',\omega_{i,j}(\tau')))^{-1}=\PTOLEMY_{ij}^{-1},
\end{equation}
where  $\simeq$ means equality up to a numerical multiplicative factor.

Projective unitary representation of the mapping class group $\MCG$ is realized as follows:
    \[
    \MCG\ni f\mapsto\FUNCTOR(\tau,f(\tau))\in
\End \vsp^{\otimes\NTR}.
    \]
    Indeed,
\[
\FUNCTOR(\tau,f(\tau))\FUNCTOR(\tau,h(\tau))=
\FUNCTOR(\tau,f(\tau))\FUNCTOR(f(\tau),f(h(\tau)))
\simeq\FUNCTOR(\tau,fh(\tau)).
\]

\subsection{Quantum discrete Liouville equation and quantum Teichm\"uller theory}
\label{sec:qdl-qtt}

Again, as in the classical case, we consider an annulus with $N$ marked points on each of its boundary components and choose a decorated ideal triangulation~$\tau_N$ shown in figure ~\ref{fig:11}.
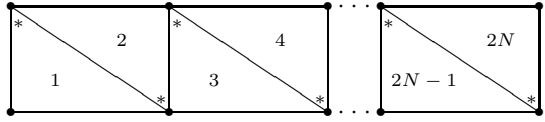
\begin{figure}[htb]
\centering
\begin{picture}(200,40)
\put(0,0){\begin{picture}(120,40)
\multiput(0,0)(0,40){2}{\line(1,0){120}}
\multiput(0,0)(60,0){3}{\line(0,1){40}}
\multiput(0,40)(60,0){2}{\line(3,-2){60}}
\multiput(0,0)(60,0){3}{\circle*{3}}
\multiput(0,40)(60,0){3}{\circle*{3}}
\scriptsize
\multiput(1,32)(60,0){2}{$*$}
\multiput(55,3)(60,0){2}{$*$}
\put(15,10){1}
\put(75,10){3}
\put(40,25){2}
\put(100,25){4}
\end{picture}}

\put(140,0){\begin{picture}(60,40)
\multiput(0,0)(0,40){2}{\line(1,0){60}}
\multiput(0,0)(60,0){2}{\line(0,1){40}}
\multiput(0,40)(60,0){1}{\line(3,-2){60}}
\multiput(0,0)(60,0){2}{\circle*{3}}
\multiput(0,40)(60,0){2}{\circle*{3}}
\scriptsize
\multiput(1,32)(60,0){1}{$*$}
\multiput(55,3)(60,0){1}{$*$}
\put(4,10){$2N-1$}
\put(40,25){$2N$}
\end{picture}}

\multiput(125,0)(5,0){3}{\circle*{1}}
\multiput(125,40)(5,0){3}{\circle*{1}}
\end{picture}
\caption{A decorated ideal triangulation of an annulus with $N$ marked points on each of the boundary components. The leftmost and the rightmost vertical edges are identified.}\label{fig:11}
\end{figure}
Equivalently, $\tau_N$ can be thought of as an infinite triangulated strip, where the triangles are numerated according to figure~\ref{fig:11}  with the periodicity condition
\[
\bar\tau_N(n+2N)=\bar\tau_N(n),\quad\forall n\in\INTEGERS.
\]
Recall that $D^{n/N}$, $n\in\INTEGERS$, is the mapping class twisting the top boundary component with respect to the bottom one  by the angle  $2\pi n/N$ so that the marked points on the top component are cyclically translated by $n$ spacings. When $n=N$ we get a pure Dehn twist
$D^{N/N}=D$. Clearly,
\[
D^{m/N}\circ D^{n/N}=D^{(m+n)/N}.
\]
From figure~\ref{fig:22}
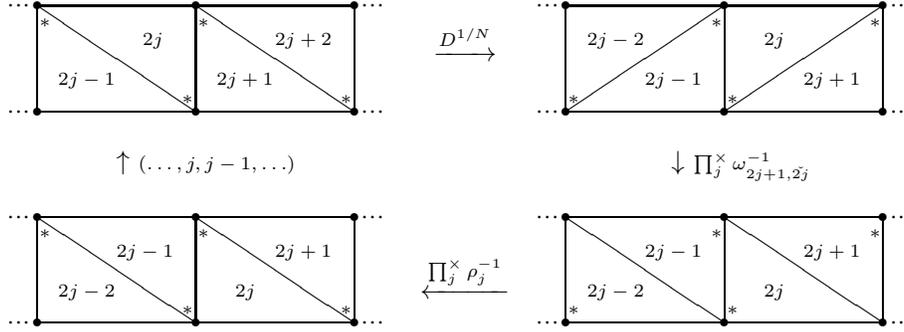
\begin{figure}[htb]
\centering
\begin{picture}(320,120)
\put(0,80){\begin{picture}(120,40)
\multiput(0,0)(0,40){2}{\line(1,0){120}}
\multiput(0,0)(60,0){3}{\line(0,1){40}}
\multiput(0,40)(60,0){2}{\line(3,-2){60}}
\multiput(0,0)(60,0){3}{\circle*{3}}
\multiput(0,40)(60,0){3}{\circle*{3}}
\scriptsize
\multiput(1,32)(60,0){2}{$*$}
\multiput(55,3)(60,0){2}{$*$}
\put(8,10){$2j-1$}
\put(68,10){$2j+1$}
\put(40,25){$2j$}
\put(90,25){$2j+2$}
\multiput(-4,0)(-3,0){3}{\circle*{1}}
\multiput(-4,40)(-3,0){3}{\circle*{1}}
\multiput(124,0)(3,0){3}{\circle*{1}}
\multiput(124,40)(3,0){3}{\circle*{1}}
\end{picture}}
\put(200,80){\begin{picture}(120,40)
\multiput(0,0)(0,40){2}{\line(1,0){120}}
\multiput(0,0)(60,0){3}{\line(0,1){40}}
\multiput(0,0)(60,0){2}{\line(3,2){60}}
\multiput(0,0)(60,0){3}{\circle*{3}}
\multiput(0,40)(60,0){3}{\circle*{3}}
\scriptsize
\multiput(1,3)(60,0){2}{$*$}
\multiput(55,32)(60,0){2}{$*$}
\put(8,25){$2j-2$}
\put(75,25){$2j$}
\put(30,10){$2j-1$}
\put(90,10){$2j+1$}
\multiput(-4,0)(-3,0){3}{\circle*{1}}
\multiput(-4,40)(-3,0){3}{\circle*{1}}
\multiput(124,0)(3,0){3}{\circle*{1}}
\multiput(124,40)(3,0){3}{\circle*{1}}
\end{picture}}
\put(0,0){\begin{picture}(120,40)
\multiput(0,0)(0,40){2}{\line(1,0){120}}
\multiput(0,0)(60,0){3}{\line(0,1){40}}
\multiput(0,40)(60,0){2}{\line(3,-2){60}}
\multiput(0,0)(60,0){3}{\circle*{3}}
\multiput(0,40)(60,0){3}{\circle*{3}}
\scriptsize
\multiput(1,32)(60,0){2}{$*$}
\multiput(55,3)(60,0){2}{$*$}
\put(8,10){$2j-2$}
\put(75,10){$2j$}
\put(30,25){$2j-1$}
\put(90,25){$2j+1$}
\multiput(-4,0)(-3,0){3}{\circle*{1}}
\multiput(-4,40)(-3,0){3}{\circle*{1}}
\multiput(124,0)(3,0){3}{\circle*{1}}
\multiput(124,40)(3,0){3}{\circle*{1}}
\end{picture}}
\put(200,0){\begin{picture}(120,40)
\multiput(0,0)(0,40){2}{\line(1,0){120}}
\multiput(0,0)(60,0){3}{\line(0,1){40}}
\multiput(0,40)(60,0){2}{\line(3,-2){60}}
\multiput(0,0)(60,0){3}{\circle*{3}}
\multiput(0,40)(60,0){3}{\circle*{3}}
\scriptsize
\multiput(1,3)(60,0){2}{$*$}
\multiput(55,32)(60,0){2}{$*$}
\put(8,10){$2j-2$}
\put(75,10){$2j$}
\put(30,25){$2j-1$}
\put(90,25){$2j+1$}
\multiput(-4,0)(-3,0){3}{\circle*{1}}
\multiput(-4,40)(-3,0){3}{\circle*{1}}
\multiput(124,0)(3,0){3}{\circle*{1}}
\multiput(124,40)(3,0){3}{\circle*{1}}
\end{picture}}
\put(150,90){$\stackrel{D^{1/N}}{\overrightarrow{\phantom{D^{1/N}}}}$}
\put(145,0){$\stackrel{\prod^{\times}_{j}
\rho^{-1}_j}{\overleftarrow{\phantom{\prod^{\times}_{j}
\rho^{-1}_j}}}$}
\put(240,58){$\downarrow$ {\scriptsize $\prod^{\times}_{j}
\omega_{2j+1,\check{2j}}^{-1}$}}
\put(30,58){$\uparrow$ {\scriptsize $(\ldots,j,j-1,\ldots)$}}
\end{picture}
\caption{Transformation $D^{1/N}$ as a morphism in the groupoid of ideal triangulations.}\label{fig:22}
\end{figure}
it follows that the quantum realization of the transformation $D^{1/N}$ has the form
\begin{equation}\label{eq:d-1/N}
\FUNCTOR\left(\tau_N,D^{1/N}(\tau_N)\right)\simeq \DEHN^{1/N}
\equiv
\zeta^{-N-6/N}\PERMUTE_{(\ldots j,{j+1}\ldots)}
\prod_{k=1}^{2N}\ROTATE_k
\prod_{l=1}^N\PTOLEMY_{2l+1,\check{2l}},
\end{equation}
where the normalization factor is chosen in accordance with the standard normalization of Dehn twists in quantum Teichm\"uller theory.
We define
\[
\DEHN^{n/N}\equiv(\DEHN^{1/N})^{n},\quad \forall n\in\INTEGERS.
\]
Consider the following faithful reducible realization of the observable algebra
$\OBALG$ in $L^2(\REALS^{2N})$:
\[
\INCL(\ff_j)=\left\{
\begin{array}{cl}
\MOM_j+\MOM_{j-1},&\mathrm{if}\ j=0\pmod2;\\
\POS_j+\POS_{j-1},&\mathrm{otherwise},
\end{array}\right.
\]
\[
\INCL(\SHFL)=\zeta^{N+6/N}e^{\IMUN2\pi\sum_{j=1}^N\MOM_{2j}\MOM_{2j+1}}
\prod_{k=1}^{2N}\ROTATE_k^{-1}
\PERMUTE_{(\ldots,l,{l-1},\ldots)}.
\]
\begin{theorem}[\cite{fk}]
One has the following equality
\begin{equation}\label{eq:u-d}
\INCL(\LC)=\DEHN^{-1/N}.
\end{equation}
\end{theorem}

\subsubsection{Working in another triangulation}

Here we consider a special decorated ideal triangulation, where the Dehn twist $\DEHN$ is represented by a single $\PTOLEMY$-operator. It happens so that with respect to this triangulation all operators $\DEHN^{n/N}$, $0<n<N$, are represented in terms of product of only $N+1$ $\PTOLEMY$-operators.

Consider a decorated ideal triangulations of the form
\[
\tau_{n:N}\equiv \ELIM_{n+1}\circ\ELIM_{n+2}\circ\cdots\circ\ELIM_N
(\tau_N),\quad 1\le n<N,
\]
where transformations $\ELIM_n$ are defined in figure~\ref{fig:33}.
\begin{figure}[htb]
\centering
\begin{picture}(320,120)
\put(200,80){\begin{picture}(120,40)
\multiput(0,0)(0,40){2}{\line(1,0){120}}
\multiput(0,0)(60,0){3}{\line(0,1){40}}
\multiput(0,0)(60,0){2}{\line(3,2){60}}
\multiput(0,0)(60,0){3}{\circle*{3}}
\multiput(0,40)(60,0){3}{\circle*{3}}
\scriptsize
\multiput(1,3)(60,0){2}{$*$}
\multiput(55,32)(60,0){2}{$*$}
\put(5,25){$2n-2$}
\put(75,25){$2n$}
\put(30,10){$2n-1$}
\put(97,10){$1$}
\multiput(-4,0)(-3,0){3}{\circle*{1}}
\multiput(-4,40)(-3,0){3}{\circle*{1}}
\multiput(124,0)(3,0){3}{\circle*{1}}
\multiput(124,40)(3,0){3}{\circle*{1}}
\end{picture}}
\put(0,80){\begin{picture}(120,40)
\multiput(0,0)(0,40){2}{\line(1,0){120}}
\multiput(0,0)(120,0){2}{\line(0,1){40}}
\multiput(0,0)(60,0){2}{\line(3,2){60}}
\put(0,0){\line(3,1){120}}
\multiput(0,0)(60,0){3}{\circle*{3}}
\multiput(0,40)(60,0){3}{\circle*{3}}
\scriptsize
\multiput(1,3)(60,0){1}{$*$}
\multiput(115,32)(60,0){1}{$*$}
\put(58,35){$*$}
\put(58,2){$*$}
\put(5,25){$2n-2$}
\put(45,5){$2n$}
\put(46,25){$2n-1$}
\put(97,10){$1$}
\multiput(-4,0)(-3,0){3}{\circle*{1}}
\multiput(-4,40)(-3,0){3}{\circle*{1}}
\multiput(124,0)(3,0){3}{\circle*{1}}
\multiput(124,40)(3,0){3}{\circle*{1}}
\end{picture}}
\put(0,0){\begin{picture}(120,40)
\multiput(0,0)(0,40){2}{\line(1,0){120}}
\multiput(0,0)(120,0){2}{\line(0,1){40}}
\put(0,0){\line(3,1){120}}
\qbezier(0,40)(10,10)(120,40)
\qbezier(0,0)(110,30)(120,0)
\multiput(0,0)(60,0){3}{\circle*{3}}
\multiput(0,40)(60,0){3}{\circle*{3}}
\scriptsize
\multiput(1,3)(60,0){1}{$*$}
\multiput(115,32)(60,0){1}{$*$}
\put(58,35){$*$}
\put(58,2){$*$}
\put(5,15){$2n-2$}
\put(80,5){$2n$}
\put(23,30){$2n-1$}
\put(110,15){$1$}
\multiput(-4,0)(-3,0){3}{\circle*{1}}
\multiput(-4,40)(-3,0){3}{\circle*{1}}
\multiput(124,0)(3,0){3}{\circle*{1}}
\multiput(124,40)(3,0){3}{\circle*{1}}
\end{picture}}
\put(143,93){$\stackrel{
\omega_{\hat{2n},2n-1}}{\overleftarrow{\phantom{\omega_{2n,2n-1}}}}$}
\put(30,57){$\downarrow$ {\scriptsize$
\omega_{2n-1,\check{2n-2}}^{-1}\times
\omega_{2n,\check{1}}^{-1}$}}
\put(160,55){\scriptsize$\ELIM_n$}\put(180,70){\vector(-2,-1){40}}
\end{picture}
\caption{Transformation $\ELIM_n$
reduces  $\tau_{n}$ to $\tau_{n-1}$ and two triangles attached to the boundary.}\label{fig:33}
\end{figure}
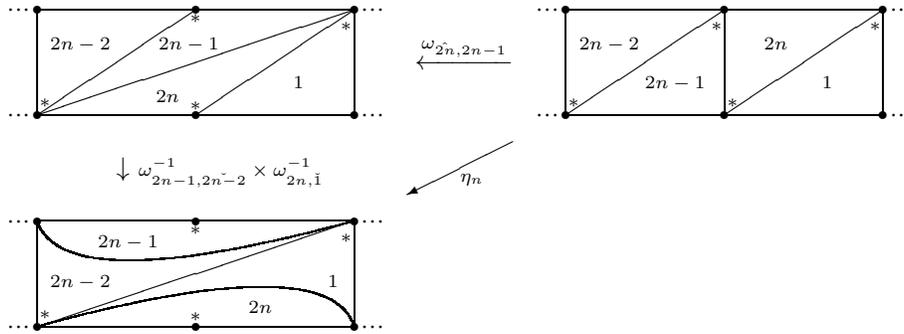
Transformation $\ELIM_n$ acts non-trivially only in the minimal annular part of $\tau_{n:N}$, which by itself is nothing else but the triangulated annulus $\tau_n$.

Quantum realizations of any element $\varphi$ of the (extended) mapping class group with respect to decorated ideal triangulations $\tau$ and $\tau'$ are conjugated to each other by operator
 $\FUNCTOR\left(\tau,\tau'\right)$:
\begin{multline*}
\FUNCTOR\left(\tau',\varphi(\tau')\right)\simeq
\FUNCTOR\left(\tau',\varphi(\tau)\right)
\FUNCTOR\left(\varphi(\tau),\varphi(\tau')\right)\\
\simeq\FUNCTOR\left(\tau',\tau\right)
\FUNCTOR\left(\tau,\varphi(\tau)\right)
\FUNCTOR\left(\tau,\tau'\right)=
\Ad\left(\FUNCTOR\left(\tau',\tau\right)\right)
\FUNCTOR\left(\tau,\varphi(\tau)\right)
\end{multline*}
In particular, from figure~\ref{fig:33} it follows that
\begin{equation}\label{eq:wn}
\FUNCTOR(\tau_N,\tau_{n:N})\simeq\SIM(n)\equiv\prod_{N\ge j> n}
(\PTOLEMY_{\hat{2j},2j-1}\PTOLEMY_{2j-1,\check{2j-2}}^{-1}
\PTOLEMY_{2j,\check{1}}^{-1}).
\end{equation}
We would like to find the realization of the transformations
$D^{n/N}$ with respect to decorated ideal triangulation~$\tau_{1:N}$:
\[
\FUNCTOR\left(\tau_{1:N},D^{n/N}(\tau_{1:N})\right)\simeq
\tilde\DEHN^{n/N}\equiv\SIM(1)^{-1}\DEHN^{n/N}\SIM(1)
\]
\begin{proposition}\label{prop:1}
Formula
\begin{multline}\label{eq:tilde-n}
\SIM(n)^{-1}\DEHN^{-1/N}\SIM(n)=\zeta^{n-1+6/N}\prod_{j=1}^{n-1}
\PTOLEMY_{2j+1,\check{2j}}^{-1}\prod_{N>k>n}
\PTOLEMY_{2N-1,2k-1}\\
\times\PTOLEMY_{2N-1,\check{2n}}\PTOLEMY_{1,2N-1}^{-1}
\PTOLEMY_{\hat{2n-1},2n}\prod_{l=1}^{2n-2}\ROTATE_l^{-1}\
\PERMUTE_{\prod_{m=n}^{N}(2m-1,2m)}
\PERMUTE_{(\ldots,s,s-1,\ldots)}
\end{multline}
holds true for $1\le n<N$. In particular, when $n=1$,
\begin{equation}\label{eq:tilde-1}
\tilde\DEHN^{-1/N}=\zeta^{6/N}
\prod_{N>j>1}\PTOLEMY_{2N-1,2j-1}\PTOLEMY_{2N-1,\check2}
\PTOLEMY_{1,2N-1}^{-1}\PTOLEMY_{\hat1,2}\PERMUTE_{(\ldots,2k+1,2k-1,\ldots)}
\end{equation}
\end{proposition}
\begin{proof}
The proof is by induction in $n$. First, let us check that equation~\eqref{eq:tilde-n} holds true at $n=N-1$. We write
\begin{multline*}
\zeta^{-N-6/N}\SIM(N-1)^{-1}\DEHN^{-1/N}\SIM(N-1)=
\underline{\PTOLEMY_{2N,\check{1}}}
\PTOLEMY_{2N-1,\check{2N-2}}\underline{\PTOLEMY_{\hat{2N-1},2N}^{-1}
\PTOLEMY_{2N,\check1}^{-1}}\\
\times\prod_{j=1}^{N-2}\PTOLEMY_{2j+1,\check{2j}}^{-1}
\PTOLEMY_{\check{2N-2},\hat{2N-3}}^{-1}
\PTOLEMY_{\check{2N},\hat{2N-1}}^{-1}
\prod_{k=1}^{2N}\ROTATE_k^{-1}\
\PERMUTE_{(\ldots,l,{l-1},\ldots)}
\end{multline*}
where we canceled one pair of $\PTOLEMY$-operators.
Applying now the pentagon equation to the underlined fragment and slightly reshuffling the commuting terms, we create another pair of $\PTOLEMY$-operators:
\begin{multline*}
=
\prod_{j=1}^{N-2}\PTOLEMY_{2j+1,\check{2j}}^{-1}
\PTOLEMY_{2N-1,\check{2N-2}}\PTOLEMY_{\hat{2N-1},\check1}^{-1}
\PTOLEMY_{\hat{2N-3},2N-2}\\
\times
\PTOLEMY_{\hat{2N-3},2N-2}^{-1}\PTOLEMY_{\check{2N-2},\hat{2N-3}}^{-1}
\PTOLEMY_{\hat{2N-1},2N}^{-1}\PTOLEMY_{\check{2N},\hat{2N-1}}^{-1}
\prod_{k=1}^{2N}\ROTATE_k^{-1}\
\PERMUTE_{(\ldots,l,{l-1},\ldots)}
\end{multline*}
finally, applying twice the inversion relation to eliminate four
$\PTOLEMY$'s in the second line, we arrive at the desired result.

Now, assuming that formula~\eqref{eq:tilde-n} holds true for
 $1<n+1<N$, we prove that it holds also for $n$:
\begin{multline*}
\zeta^{-n-6/N}\Ad\left(\SIM(n)^{-1}\right)\DEHN^{-1/N}\\
=\zeta^{-n-6/N}\Ad(\PTOLEMY_{2n+2,\check{1}}
\PTOLEMY_{2n+1,\check{2n}}\PTOLEMY_{\hat{2n+2},2n+1}^{-1})
\circ\Ad\left(\SIM(n+1)^{-1}\right)\DEHN^{-1/N}\\
=\PTOLEMY_{2n+2,\check{1}}
\underline{\PTOLEMY_{2n+1,\check{2n}}\PTOLEMY_{\hat{2n+2},2n+1}^{-1}
\PTOLEMY_{2n+1,\check{2n}}^{-1}}\prod_{j=1}^{n-1}
\PTOLEMY_{2j+1,\check{2j}}^{-1}\\
\times\prod_{N>k>n+1}
\PTOLEMY_{2N-1,2k-1}\PTOLEMY_{2N-1,\check{2n+2}}\PTOLEMY_{1,2N-1}^{-1}
\PTOLEMY_{\hat{2n+1},2n+2}\\
\times\PTOLEMY_{\hat{2n+2},\check{2n}}
\PTOLEMY_{\check{2n},\hat{2n-1}}^{-1}
\PTOLEMY_{2n+2,\check{2N-1}}^{-1}
\prod_{l=1}^{2n}\ROTATE_l^{-1}\
\PERMUTE_{\prod_{m=n+1}^{N}(2m-1,2m)}
\PERMUTE_{(\ldots,s,s-1,\ldots)}
\end{multline*}
applying the pentagon relation to the underlined fragment,
\begin{multline*}
=\prod_{j=1}^{n-1}
\PTOLEMY_{2j+1,\check{2j}}^{-1}\prod_{N>k>n+1}
\PTOLEMY_{2N-1,2k-1}
\\
\times\PTOLEMY_{2n+2,\check{1}}
\PTOLEMY_{\hat{2n+2},\check{2n}}^{-1}
\underline{\PTOLEMY_{\hat{2n+1},2n+2}^{-1}
\PTOLEMY_{2n+2,\check{2N-1}}}\PTOLEMY_{1,2N-1}^{-1}
\underline{\PTOLEMY_{\hat{2n+1},2n+2}}\\
\times\PTOLEMY_{\hat{2n+2},\check{2n}}
\PTOLEMY_{\check{2n},\hat{2n-1}}^{-1}
\PTOLEMY_{2n+2,\check{2N-1}}^{-1}
\prod_{l=1}^{2n}\ROTATE_l^{-1}\
\PERMUTE_{\prod_{m=n+1}^{N}(2m-1,2m)}
\PERMUTE_{(\ldots,s,s-1,\ldots)}
\end{multline*}
again applying the pentagon equation and hiding one $\PTOLEMY$ into the product over $k$,
\begin{multline*}
=\prod_{j=1}^{n-1}
\PTOLEMY_{2j+1,\check{2j}}^{-1}\prod_{N>k>n}
\PTOLEMY_{2N-1,2k-1}
\PTOLEMY_{2n+2,\check{1}}
\underline{\PTOLEMY_{{2n},{2n+2}}^{-1}
\PTOLEMY_{2n+2,\check{2N-1}}}
\underline{\underline{\PTOLEMY_{1,2N-1}^{-1}}}\\
\times
\underline{\PTOLEMY_{{2n},{2n+2}}}
\PTOLEMY_{\check{2n},\hat{2n-1}}^{-1}
\underline{\underline{\PTOLEMY_{2N-1,\check{2n+2}}^{-1}}}
\prod_{l=1}^{2n}\ROTATE_l^{-1}\
\PERMUTE_{\prod_{m=n+1}^{N}(2m-1,2m)}
\PERMUTE_{(\ldots,s,s-1,\ldots)}
\end{multline*}
two more pentagon relations with subsequent cancelation of two pairs of $\PTOLEMY$'s
\begin{multline*}
=\prod_{j=1}^{n-1}
\PTOLEMY_{2j+1,\check{2j}}^{-1}\prod_{N>k>n}
\PTOLEMY_{2N-1,2k-1}\PTOLEMY_{{2N-1},\check{2n}}
\PTOLEMY_{1,2N-1}^{-1}
\PTOLEMY_{\check{2n},\hat{2n-1}}^{-1}
\\
\times
\prod_{l=1}^{2n}\ROTATE_l^{-1}\
\PERMUTE_{\prod_{m=n+1}^{N}(2m-1,2m)}
\PERMUTE_{(\ldots,s,s-1,\ldots)}
\end{multline*}
finally, application of the inversion relation to the last $\PTOLEMY$ gives
 equation~\eqref{eq:tilde-n}. Formula~\eqref{eq:tilde-1} is a particular case of \eqref{eq:tilde-n} corresponding to $n=1$.
\end{proof}
\begin{proposition}\label{prop:2}
Formula
\begin{multline}\label{eq:power-any}
\tilde\DEHN^{(n-N)/N}=\zeta^{6(N-n)/N}
\prod_{n\ge j>1}\PTOLEMY_{2n+1,2j-1}\PTOLEMY_{2n+1,\check2}
\PTOLEMY_{1,2n+1}^{-1}\PTOLEMY_{\hat1,2}\\
\times
\prod_{n<k<N}\PTOLEMY_{1,2k+1}^{-1}
\left(\PERMUTE_{(\ldots,2l-1,2l+1,\ldots)}\right)^n
\end{multline}
holds true for $1\le n<N$, while
\begin{equation}\label{eq:power-N}
\tilde\DEHN=\zeta^{-6}\PTOLEMY_{\check1,\hat2}^{-1}.
\end{equation}
\end{proposition}
\begin{proof}
Again we use induction in
$n$. Equation~\eqref{eq:power-any} at $n=N-1$ coincides with equation~\eqref{eq:tilde-1}. Assuming the statement is true for
$1<n+1<N$, we derive it for $n$:
\begin{multline*}
\zeta^{6(n-N)/N}\tilde\DEHN^{(n-N)/N}=
\zeta^{6(n-N)/N}\tilde\DEHN^{-1/N}\tilde\DEHN^{(n+1-N)/N}=
\prod_{N>j>1}\PTOLEMY_{2N-1,2j-1}\\
\times
\PTOLEMY_{2N-1,\check2}
\PTOLEMY_{1,2N-1}^{-1}\underline{\PTOLEMY_{\hat1,2}}
\prod_{n\ge l>1}\PTOLEMY_{2n+1,2l-1}
\underline{\PTOLEMY_{2n+1,1}\PTOLEMY_{2n+1,\check2}}
\PTOLEMY_{2N-1,2n+1}^{-1}\\
\times\PTOLEMY_{\hat{2N-1},2}
\prod_{n+1<k<N}\PTOLEMY_{2N-1,2k-1}^{-1}
\left(\PERMUTE_{(\ldots,2m-1,2m+1,\ldots)}\right)^{n}
\end{multline*}
applying the pentagon relation to the underlined fragment with subsequent reshuffling of commuting terms,
\begin{multline*}
=\prod_{N>j>n+1}\PTOLEMY_{2N-1,2j-1}
\underline{\PTOLEMY_{2N-1,2n+1}}
\prod_{n\ge l>1}
(\underline{\PTOLEMY_{2N-1,2l-1}\PTOLEMY_{2n+1,2l-1}})\\
\times
\underline{\PTOLEMY_{2N-1,\check2}\PTOLEMY_{2n+1,\check2}
\PTOLEMY_{1,2N-1}^{-1}
\PTOLEMY_{2N-1,2n+1}^{-1}}\PTOLEMY_{\hat1,2}\PTOLEMY_{\hat{2N-1},2}\\
\times
\prod_{n+1<k<N}\PTOLEMY_{2N-1,2k-1}^{-1}
\left(\PERMUTE_{(\ldots,2m-1,2m+1,\ldots)}\right)^{n}
\end{multline*}
applying consecutively the pentagon relation $n+1$ times, starting from the product on $l$
\begin{multline*}
=\prod_{N>j>n+1}\PTOLEMY_{2N-1,2j-1}
\prod_{n\ge l>1}\PTOLEMY_{2n+1,2l-1}\PTOLEMY_{2n+1,\check2}
\PTOLEMY_{1,2n+1}^{-1}
\\
\times
\underline{\PTOLEMY_{1,2N-1}^{-1}
\PTOLEMY_{\hat2,1}\PTOLEMY_{\hat{2},2N-1}}
\prod_{n+1<k<N}\PTOLEMY_{2N-1,2k-1}^{-1}
\left(\PERMUTE_{(\ldots,2m-1,2m+1,\ldots)}\right)^{n}
\end{multline*}
one more pentagon relation with little reshuffling,
\begin{multline*}
=
\prod_{n\ge j>1}\PTOLEMY_{2n+1,2j-1}\PTOLEMY_{2n+1,\check2}
\PTOLEMY_{1,2n+1}^{-1}\PTOLEMY_{\hat1,2}
\\
\times
\prod_{N>l>n+1}\PTOLEMY_{2N-1,2l-1}\PTOLEMY_{1,2N-1}^{-1}
\prod_{n+1<k<N}\PTOLEMY_{2N-1,2k-1}^{-1}
\left(\PERMUTE_{(\ldots,2m-1,2m+1,\ldots)}\right)^{n}
\end{multline*}
finally, consecutive application of the pentagon relation  $N-n-2$ more times in the second line proves equation~\eqref{eq:power-any}.

Proof of equation~\eqref{eq:power-N} is now a relatively simple task:
\begin{multline*}
\zeta^{-6}\tilde\DEHN^{-1}=\zeta^{-6}\tilde\DEHN^{-N/N}=
\zeta^{-6}\tilde\DEHN^{(1-N)/N}\tilde\DEHN^{-1/N}
=\PTOLEMY_{3\check2}
\PTOLEMY_{13}^{-1}\underline{\PTOLEMY_{\hat12}\PTOLEMY_{2\check1}}
\PTOLEMY_{31}^{-1}\PTOLEMY_{\hat32}\\
=\PTOLEMY_{3\check2}\PTOLEMY_{13}^{-1}\PTOLEMY_{3\check2}^{-1}
\PTOLEMY_{13}\zeta\PERMUTE_{(\hat12\check1)}
=\PTOLEMY_{1\check2}^{-1}
\zeta\PERMUTE_{(\hat12\check1)}=\PTOLEMY_{\check1\hat2}.
\end{multline*}
\end{proof}

\subsubsection{Description in terms of variables of the discrete Liouville equation}
Define operators
\[
\GENI_{j,k}\equiv\sum_{l=j+1}^k\ff_{j}.
\]
Notice that the operator, corresponding to transformation~\eqref{eq:wn}, can be written in the form
\[
\SIM(n)=\INCL(\SIMF(n))\SIME(n),
\]
where
\[
\SIMF(n)=\prod_{N\ge j>n}\left(\QDILOG(\GEN_{2j})\QDILOGI(\GEN_{2j-1})
\QDILOGI(\GENI_{2j,2N+1})\right),\quad
\QDILOGI(x)\equiv (\QDILOG(x))^{-1},
\]
\[
\SIME(n)=\prod_{N\ge j>n}\left(e^{-\IMUN2\pi\POS_{2j}\POS_{2j-1}}
e^{\IMUN2\pi(\MOM_{2j-1}\MOM_{2j-2}+\MOM_{2j}\MOM_{1})}\right).
\]
Similarly, formula~\eqref{eq:tilde-n} can be written as
\[
\Ad(\SIM(n)^{-1})\DEHN^{-1/N}=\Ad(\SIME(n)^{-1})\INCL\left(
\Ad(\SIMF(n)^{-1})\LC\right).
\]
The following propositions can be proved in similar manner as propositions~\ref{prop:1} and \ref{prop:2}.
\begin{proposition}
Formula
\begin{multline*}
\Ad(\SIMF(n)^{-1})\LC={\tilde{\zeta}}^{n-N-1}
\prod_{j=2}^n\QDILOGI(\GEN_{2j-1})\prod_{N>k>n}
\QDILOG(\GENI_{2k-1,2N-1})\\
\times
\QDILOG(\GENI_{2n,2N-1})\QDILOGI(\GENI_{2N-1,2N+1})
\QDILOG(\GEN_{2n})e^{\IMUN\pi\sum_{l=n}^N\GEN_{2l}^2}\SHFL,\quad
\tilde\zeta\equiv\zeta e^{\IMUN\pi/6},
\end{multline*}
holds true for $1\le n<N$.
In particular,
\begin{multline*}
\LCT\equiv \Ad(\SIMF(1)^{-1})\LC={\tilde{\zeta}}^{-N}
\prod_{N>k>1}
\QDILOG(\GENI_{2k-1,2N-1})\\
\times
\QDILOG(\GENI_{2,2N-1})\QDILOGI(\GENI_{2N-1,2N+1})
\QDILOG(\GEN_{2})e^{\IMUN\pi\sum_{l=1}^N\GEN_{2l}^2}\SHFL.
\end{multline*}
\end{proposition}
\begin{proposition}
Formula
\begin{multline*}
\LCT^{N-n}=\tilde\zeta^{(n-N)N}
\prod_{n\ge j>1}\QDILOG(\GENI_{2j-1,2n+1})\QDILOG(\GENI_{2,2n+1})
\QDILOGI(\GENI_{2n+1,2N+1})\QDILOG(\GEN_2)\\
\times
\prod_{n<k<N}\QDILOGI(\GENI_{2k+1,2N+1})
\left( e^{\IMUN\pi\sum_{l=1}^N\GEN_{2l}^2}\SHFL\right)^{N-n}
\end{multline*}
holds true for $1\le n<N$, while
\[
\LCT^{N}={\tilde{\zeta}}^{1-N^2}
\QDILOGI(\GENI_{2,2N+1})e^{-\IMUN\pi\GEN_2^2}
\left( e^{\IMUN\pi\sum_{l=1}^N\GEN_{2l}^2}\SHFL\right)^{N}.
\]
\end{proposition}

\end{document}